\documentclass[10pt,a4,leqno]{amsart}
\usepackage{amsthm}
\usepackage{amsmath}
\usepackage{amssymb}
\usepackage{bm}
\usepackage{graphicx}
\usepackage[cmtip,arrow]{xy}
\usepackage{pb-diagram,pb-xy}
\usepackage{bbm}
\usepackage{color}
\allowdisplaybreaks

\makeatletter

\@addtoreset{equation}{section}
\makeatother

\newtheoremstyle{my theoremstyle}
{1.0em}                    
    {1.0em}                    
    {\itshape}                   
    {}                           
    {\scshape}                   
    {.}                          
    {.5em}                       
    {}  

\newtheoremstyle{dfn}
{1.0em}                    
    {1.0em}                    
    {}                   
    {}                           
    {\scshape}                   
    {.}                          
    {.5em}                       
    {}  
    
\theoremstyle{my theoremstyle}
   \newtheorem{thm}{Theorem}[section]
   \newtheorem{lem}[thm]{Lemma}
   \newtheorem{prop}[thm]{Proposition}
   \newtheorem{cor}[thm]{Corollary}
\theoremstyle{dfn}

\theoremstyle{remark}   
   
   \newtheorem{rem}[thm]{{\scshape Remark}}

\newcommand{\C}{\mathbb{C}}
\newcommand{\Q}{\mathbb{Q}}
\newcommand{\Z}{\mathbb{Z}}

\renewcommand{\a}{\alpha}
\renewcommand{\b}{\beta}
\renewcommand{\c}{\gamma}
\newcommand{\e}{\varepsilon}
\renewcommand{\d}{\delta}
\renewcommand{\k}{\kappa}
\renewcommand{\l}{\lambda}
\newcommand{\khat}{\widehat{\k^\times}}

\newcommand{\ol}[1]{\overline{#1}}

\newcommand{\hF}[5]{{}_{#1}F_{#2}\left({#3 \atop #4}; #5 \right)}

\newcommand{\FA}[4]{F_A^{(#1)}\left({#2\atop#3};#4\right)}
\newcommand{\FB}[4]{F_B^{(#1)}\left({#2\atop#3};#4\right)}
\newcommand{\FC}[4]{F_C^{(#1)}\left({#2\atop#3};#4\right)}
\newcommand{\FD}[4]{F_D^{(#1)}\left({#2\atop#3};#4\right)}

\newcommand{\FAA}[3]{F_2\left({#1\atop#2};#3\right)}
\newcommand{\FBB}[3]{F_3\left({#1\atop#2};#3\right)}
\newcommand{\FCC}[3]{F_4\left({#1\atop#2};#3\right)}

\begin{document}
\title[]{Sum representations of Appell-Lauricella functions over finite fields using confluent hypergeometric functions and their applications}
\author{Akio Nakagawa}
\date{\today}
\address{Institute of Science and Engineering,
Faculty of Mathematics and Physics,
Kanazawa University,
Kakuma, Kanazawa, Ishikawa, 920-1192, Japan}
\email{akio.nakagawa.math@icloud.com}

\maketitle

\begin{abstract}
We prove sum representations of Appell-Lauricella functions over a finite field using confluent hypergeometric functions over the finite field.
As an application, we also prove transformation formulas, summation formulas and reduction formulas for Appell-Lauricella functions over the finite field.
\end{abstract}


\section{Introduction}
Over $\C$, one-variable hypergeometric functions ${}_mF_n$ are defined by the power series
$$ \hF{m}{n}{a_1,\dots,a_m}{b_1,\dots,b_n}{z} := \sum_{k=0}^\infty \dfrac{\prod_{i=1}^m (a_i)_k}{(1)_k \prod_{i=1}^n (b_i)_k}z^k.$$ 
Here, $a_i, b_i \in \C$ ($b_i \not\in \Z_{\leq 0}$) are parameters and $(a)_k = \Gamma(a+k)/\Gamma(a)$ is the Pochhammer symbol.
The functions ${}_mF_n$ are called confluent when $m \neq n+1$.
Lauricella's hypergeometric functions $F_A^{(n)}, F_B^{(n)}, F_C^{(n)}$ and $F_D^{(n)}$ with $n$ variables (Appell's functions $F_2, F_3, F_4$ and $F_1$ respectively, when $n=2$) are generalizations of the functions ${}_2F_1$ (cf. \cite{Matsumoto}). 
For example,
$$\FA{n}{a ; b_1, \dots, b_n}{c_1, \dots, c_n}{z_1, \dots, z_n} := \sum_{m_1,\dots,m_n =0}^\infty \dfrac{(a)_{m_1+\cdots+m_n} \prod_{i=1}^n (b_i)_{m_i}}{\prod_{i=1}^n (1)_{m_i} (c_i)_{m_i}}\prod_{i=1}^n z_i^{m_i}.$$

Over a finite field $\k$, one-variable hypergeometric functions were defined by Koblitz \cite{Koblitz}, Greene \cite{Greene}, McCarthy \cite{Mc} and Fuselier-Long-Ramakrishna-Swisher-Tu \cite{FLRST} when $m=n+1$, and by Katz \cite{Katz} and Otsubo \cite{Otsubo} in general.
Appell's functions over $\k$ were defined by Li-Li-Mao \cite{LLM}, He \cite{He}, He-Li-Zhang \cite{HLZ} and Ma \cite{Ma}, Tripathi-Barman \cite{TB} and Tripathi-Saikia-Barman \cite{TSB}.
For general $n$, $F_D^{(n)}$ were defined by Frechette-Swisher-Tu \cite{FST} and He \cite{He}, and $F_A^{(n)}$ were defined by Chetry-Kalita \cite{CK}. Otsubo \cite{Otsubo} gave a definition of all the Lauricella functions.
Here, we use Otsubo's hypergeometric functions over $\k$:
\begin{align*}
&\hF{m}{n}{\a_1, \dots, \a_m}{\b_1, \dots, \b_n}{\l},\\
&\FA{n}{\a\, ;\, \b_1, \dots, \b_n}{\c_1, \dots, \c_n}{\l_1, \dots,\l_n},\quad \FB{n}{\a_1, \dots, \a_n\, ;\, \b_1, \dots, \b_n}{\c}{\l_1, \dots,\l_n},\\
&\FC{n}{\a\, ;\, \b}{\c_1, \dots, \c_n}{\l_1, \dots,\l_n},\quad \FD{n}{\a\, ;\, \b_1, \dots, \b_n}{\c}{\l_1, \dots, \l_n},
\end{align*}
where the parameters $\a, \b, \c, \a_i, \b_i$ and $\c_i$ are multiplicative characters of $\k$ and $\l, \l_i \in \k$. 
See Section 2 for the definition and \cite[Remark 2.13]{Otsubo} for the relation with other definitions.
One of the merits of Otsubo's definition is that we can treat confluent hypergeometric functions over $\k$.
Similarly to the complex case, various transformation formulas, summation formulas and reduction formulas for Appell-Lauricella functions over finite fields are known (cf. \cite{CK,FST,He,HLZ,LLM,Ma,N,TB,TB2,TSB}).


The first purpose of this paper is to prove the following sum representations using confluent hypergeometric functions for Lauricella's functions over $\k$ (Theorem \ref{FA and FC integral rep.}):  
\small
\begin{align*}
 & -g(\a) \FA{n}{\a\, ;\, \b_1,\dots,\b_n}{\c_1,\dots,\c_n}{\l_1,\dots,\l_n} = \sum_{t \in \k^\times} \psi(t) \a(t) \prod_{i=1}^n \hF{1}{1}{\b_i}{\c_i}{\l_i t},\\
 & -\dfrac{q}{g^\circ(\c)} \FB{n}{\a_1,\dots,\a_n\, ;\, \b_1,\dots,\b_n}{\c}{\l_1,\dots,\l_n}= \sum_{t\in\k^\times} \psi(-t) \ol{\c}(t) \prod_{i=1}^n \hF{2}{0}{\a_i, \b_i}{}{\dfrac{\l_i}{t}},\\
 & - g(\a^2) \FC{n}{\a\, ;\, \a\phi}{\b_1,\dots, \b_n}{\l_1, \dots, \l_n} = \sum_{t \in \k^\times} \psi(t) \a^2(t) \prod_{i=1}^n \hF{0}{1}{}{\b_i}{\dfrac{\l_i t^2}{4}}\,\, (\mathrm{char}(\k) \neq 2),\\
& g(\a)g(\b) \FC{n}{\a\, ;\, \b}{\c_1,\dots,\c_n}{\l_1,\dots,\l_n} = \sum_{s, t \in \k^\times} \psi(s+t)\a(s)\b(t) \prod_{i=1}^n \hF{0}{1}{}{\c_i}{\l_i st}.
\end{align*}
\normalsize
Here, $q=\#\k$, $\psi$ is a fixed non-trivial additive character of $\k$, $g(\a)$ and $g^\circ(\a)$ are Gauss sums, $\phi$ is the quadratic character and $\ol{\a} = \a^{-1}$.
These formulas are useful to prove certain transformation, summation and reduction formulas for Appell-Lauricella functions over $\k$.
The first and the third formulas above are finite field analogues of the integral representations (cf. \cite[p.285]{Srivastava-Karlsson} and \cite{Srivastava-Exton}) of the complex functions:
\begin{align*}
& \Gamma(a) \FA{n}{a\, ;\, b_1, \dots, b_n}{c_1,\dots,c_n}{z_1, \dots, z_n}
 = \int_0^\infty e^{-t} t^{a-1} \prod_{i=1}^n \hF{1}{1}{b_i}{c_i}{z_i t}\, dt,\\
& \Gamma(2a) \FC{n}{a\, ;\, a+\frac{1}{2}}{b_1,\dots,b_n}{z_1,\dots, z_n} = \int_0^\infty e^{-t}t^{2a-1} \prod_{i=1}^n \hF{0}{1}{}{b_i}{\dfrac{z_i t^2}{4}}\, dt.
\end{align*}
However, as far as the author knows, integral representations of Lauricella $F_B$ and $F_C$ over $\C$ corresponding to the second and the fourth formulas above are not known.
For Lauricella's $F_D$ over $\C$, there is a similar integral representation, and its analogue has already been obtained by the author \cite{N} (see also Remark \ref{rem1}).


The second purpose of this paper is to prove transformation, summation and reduction formulas for Appell-Lauricella functions over $\k$ mainly by using the sum representations above and product formulas for confluent hypergeometric functions over $\k$ due to Otsubo \cite{Otsubo} and Otsubo-Senoue \cite{Otsubo2}. 
Some of the formulas shown have analogous formulas over $\C$, while others do not.
For example, we give the following two formulas (Theorems \ref{FA = FC} and \ref{F3 and 3F2, 2F1}) under some conditions for parameters.\\
(i) When $\mathrm{char}(\k) \neq 2$ and $1+\sum_{i=1}^n \l_i \neq 0$,
\begin{align*} 
&\FC{n}{\a\, ;\, \a\phi}{\b_1\phi, \dots, \b_n\phi}{\l_1^2, \dots, \l_n^2}\\
&= \ol{\a}^2 \Big( 1 + \sum_{i=1}^n \l_i \Big) \FA{n}{\a^2\, ;\, \b_1 , \dots, \b_n}{\b_1^2, \dots, \b_n^2}{\dfrac{2\l_1}{1+\sum_{i=1}^n \l_i},\dots, \dfrac{2\l_n}{1+\sum_{i=1}^n \l_i}}.
\end{align*}
This formula is a finite field analogue of Bailey's formula (cf. \cite{Bailey})
\begin{align*}
& \FC{n}{a\, ;\, a + \frac{1}{2}}{b_1+\frac{1}{2}, \dots,b_n+\frac{1}{2}}{z_1^2,\dots,z_n^2}\\
& =\Big( 1+\sum_{i=1}^n z_i \Big)^{-2a} \FA{n}{2a\, ;\, b_1,\dots,b_n}{2b_1,\dots,2b_n}{\dfrac{2z_1}{1+\sum_{i=1}^n z_i},\dots,\dfrac{2z_n}{1+\sum_{i=1}^n z_i}}.
\end{align*} 
(ii) For $\l\neq1$, we have
\begin{align*}
& \FBB{\a_1,\a_2 \,;\, \b_1, \b_2}{\c}{\l, \dfrac{\l}{\l-1}}=\dfrac{g(\a_1\a_2\ol{\c}) g(\a_2\b_1\ol{\c}) g(\ol{\b_2})}{g(\ol{\c}) g^\circ (\a_1\a_2\b_1\ol{\c}) g^\circ(\a_2\ol{\b_2})}\\
& \hspace{70pt}\times \ol{\c}(\l) \a_2(1-\l) \hF{3}{2}{\a_1\a_2\ol{\c}, \a_2\b_1\ol{\c}, \ol{\b_2}}{\a_1\a_2\b_1\ol{\c}, \a_2\ol{\b_2}}{1-\l}.
\end{align*}
This formula is a relation between Appell's and one-variable hypergeometric functions over $\k$ whose corresponding functions over $\C$ satisfy the same ordinary differential equation of second order \cite{Vidunas}.
As a fact, those functions over $\C$ are linearly independent (it can be checked by numerical experiments) and hence, there is no analogous formula over $\C$.

This paper is constructed as follows.
In Section 2, we recall the definitions and some basic properties of hypergeometric functions over $\k$.
In Section 3, we prove formulas for Appell-Lauricella functions over $\k$.
More precisely, in Subsection 3.1, we prove the sum representations of Lauricella's functions, and in Subsection 3.2, we obtain some formulas for confluent hypergeometric functions which will be used after.
In Subsection 3.3, we prove reduction formulas, summation formulas and transformation formulas for Lauricella's $F_A$ and $F_C$, and in Subsections 3.4-3.6, we obtain reduction formulas and transformation formulas for Appell's $F_2, F_3$ and $F_4$.
In Subsection 3.7, we prove formulas between Appell's functions and ${}_4F_3$-functions, and we prove reduction formulas for Lauricella's $F_D$ in Subsection 3.8.

\section{Definitions and basic properties}
Throughout this paper, let $\k$ be a finite field with $q$ elements of characteristic $p$. 
Let $\khat={\rm Hom}(\k^\times,\ol{\Q}^\times)$ be the group of multiplicative characters of $\k$.
We write $\e, \phi \in\khat$ for the trivial character and the quadratic character, respectively, where the latter is only used when $p\neq2$.
For any $\eta\in\khat$, we set $\eta(0)=0$ and write $\ol{\eta}=\eta^{-1}$. 
Put, for $\eta\in\widehat{\k^\times}$,
\begin{equation*}
\d(\eta)=\begin{cases} 1&(\eta=\e),\\ 0&(\eta\neq\e).\end{cases}
\end{equation*}
We fix a non-trivial additive character $\psi\in{\rm Hom}(\k,\ol{\Q}^\times)$.

For $\eta,\eta_1,\dots,\eta_n  \in \widehat{\k^\times}$ ($n\geq 2$), {\it the Gauss sum} $g(\eta)$ and {\it the Jacobi sum} $j(\eta_1,\dots,\eta_n)$ are defined by 
\begin{align*}
g(\eta)&=-\sum_{x\in \k^\times} \psi(x)\eta(x)\ \in\Q(\mu_{p(q-1)}),\\
j(\eta_1,\dots,\eta_n)&=(-1)^{n-1}\sum_{\substack{x_i\in\k^\times\\ x_1+\dots+x_n=1}}\prod_{i=1}^n \eta_i(x_i)\ \in\Q(\mu_{q-1}).
\end{align*}
Note that $g(\e)=1$. 
Put $g^\circ(\eta)=q^{\d(\eta)}g(\eta)$. Then, one shows (cf. \cite[Proposition 2.2 (iii)]{Otsubo})
\begin{equation}\label{Gauss sum thm}
g(\eta)g^\circ(\ol{\eta})=\eta(-1)q.
\end{equation}
For $\eta_1,\dots,\eta_n\in\widehat{\k^\times}$, we have (cf. \cite[Proposition 2.2 (iv)]{Otsubo})
\begin{equation}\label{J=G}
j(\eta_1,\dots,\eta_n)=
\begin{cases}
\dfrac{1-(1-q)^n}{q}&(\eta_1=\cdots=\eta_n=\e),\vspace{5pt}\\
\dfrac{g(\eta_1)\cdots g(\eta_n)}{g^\circ(\eta_1\cdots\eta_n)}&({\rm otherwise}).
\end{cases}
\end{equation}
As an analogue of the Pochhammer symbol $(a)_n=\Gamma(a+n)/\Gamma(a)$, put 
\begin{align*}
(\a)_\nu=\dfrac{g(\a\nu)}{g(\a)},\ \ \ \ (\a)_\nu^\circ=\dfrac{g^\circ(\a\nu)}{g^\circ(\a)}
\end{align*}
for $\a,\ \nu\in\widehat{\k^\times}$. 
One shows 
\begin{align}\label{Poch formula}
(\a)_{\nu\mu}=(\a)_\nu(\a\nu)_\mu,\ \ \ (\a)_{\nu\mu}^\circ=(\a)_\nu^\circ(\a\nu)_\mu^\circ
\end{align}
and 
\begin{equation}
(\a)_\nu = \nu(-1) \dfrac{1}{(\ol{\a})_{\ol{\nu}}^\circ}. \label{inversion Poch}
\end{equation}
If $p \neq 2$, we will use the duplication formulas (cf. \cite[Theorem 3.10 and Corollary 3.11]{Otsubo})
\begin{equation}
g(\a^2) = \a(4) \dfrac{g(\a)g(\a\phi)}{g(\phi)},\quad g^\circ(\a^2) = \a(4) \dfrac{g(\a)^\circ g(\a\phi)}{g(\phi)}\label{dup. form. gauss}
\end{equation}
and
\begin{equation}
(\a^2)_{\nu^2} = \nu(4) (\a)_\nu (\a\phi)_\nu,\quad (\a^2)_{\nu^2}^\circ = \nu(4) (\a)_\nu^\circ (\a\phi)_\nu^\circ. \label{dup. form. Poch}
\end{equation}

For $\a_1,\dots,\a_m,\b_1,\dots,\b_n\in\widehat{\k^\times}$, the hypergeometric function over $\k$ is defined by  
\begin{equation*}
\hF{m}{n}{\a_1,\dots,\a_m}{\b_1,\dots,\b_n}{\lambda}=\dfrac{1}{1-q}\sum_{\nu\in\widehat{\k^\times}}\dfrac{(\a_1)_\nu\cdots(\a_m)_\nu}{(\e)_\nu^\circ(\b_1)_\nu^\circ\cdots(\b_n)_\nu^\circ}\nu(\lambda)\quad (\lambda\in \k).
\end{equation*}
By Otsubo \cite[Corollary 3.5 (i)]{Otsubo}, we have that if $\e \not\in\{\a_0, \ol{\a_1}\b_1, \ol{\a_2}\b_2\}$ and $\l \neq 0$, then
\begin{align}
& \Big( \prod_{i=1}^2 j(\a_i, \ol{\a_i}\b_i)\Big) \hF{3}{2}{\a_0, \a_1, \a_2}{\b_1, \b_2}{\l} \label{3F2 int.}\\
& = \sum_{s,t\in\k^\times} \ol{\a_0}( 1-\l st) \a_1(s) \ol{\a_1}\b_1(1-s) \a_2(t) \ol{\a_2}\b_2(1-t) .\nonumber
\end{align}
For $\l \in \k^\times$, one shows (cf. \cite[Corollary 3.4]{Otsubo}) that when $\a \neq \e$ or $\l \neq 1$, 
\begin{equation}
\hF{1}{0}{\a}{}{\l} = \ol{\a}(1-\l).\label{1F0 int.}
\end{equation}
In particular, $\hF{1}{0}{\a}{}{1} = 0$ when $\a\neq \e$.

For $\a,\a_1,\dots,\a_n,\b,\b_1,\dots,\b_n,\c,\c_1,\dots,\c_n\in\widehat{\k^\times}$, Lauricella's functions over $\k$ are defined as follows. For $\lambda_1,\dots,\lambda_n\in\k$,
\begin{align*}
&\FA{n}{\a\, ;\, \b_1,\dots,\b_n}{\c_1,\dots,\c_n}{\lambda_1,\dots,\lambda_n}\\
&:=\dfrac{1}{(1-q)^n}\sum_{\nu_i\in\widehat{\k^\times}}\dfrac{(\a)_{\nu_1\cdots\nu_n}(\b_1)_{\nu_1}\cdots(\b_n)_{\nu_n}}{(\c_1)_{\nu_1}^\circ\cdots(\c_n)_{\nu_n}^\circ(\e)_{\nu_1}^\circ\cdots(\e)_{\nu_n}^\circ}\prod_{i=1}^n \nu_i(\lambda_i),\\
&\FB{n}{\a_1,\dots,\a_n\, ;\, \b_1,\dots,\b_n}{\c}{\lambda_1,\dots,\lambda_n}\\
&:=\dfrac{1}{(1-q)^n}\sum_{\nu_i\in\widehat{\k^\times}}\dfrac{(\a_1)_{\nu_1}\cdots(\a_n)_{\nu_n}(\b_1)_{\nu_1}\cdots(\b_n)_{\nu_n}}{(\c)_{\nu_1\cdots\nu_n}^\circ(\e)_{\nu_1}^\circ\cdots(\e)_{\nu_n}^\circ}\prod_{i=1}^n \nu_i(\lambda_i),\\
&\FC{n}{\a\, ;\, \b}{\c_1,\dots,\c_n}{\lambda_1,\dots,\lambda_n}\\
&:=\dfrac{1}{(1-q)^n}\sum_{\nu_i\in\widehat{\k^\times}}\dfrac{(\a)_{\nu_1\cdots\nu_n}(\b)_{\nu_1\cdots\nu_n}}{(\c_1)_{\nu_1}^\circ\cdots(\c_n)_{\nu_n}^\circ(\e)_{\nu_1}^\circ\cdots(\e)_{\nu_n}^\circ}\prod_{i=1}^n \nu_i(\lambda_i),\\
&\FD{n}{\a\, ;\, \b_1,\dots,\b_n}{\c}{\lambda_1,\dots,\lambda_n}\\
&:=\dfrac{1}{(1-q)^n}\sum_{\nu_i\in\widehat{\k^\times}}\dfrac{(\a)_{\nu_1\cdots\nu_n}(\b_1)_{\nu_1}\cdots(\b_n)_{\nu_n}}{(\c)_{\nu_1\cdots\nu_n}^\circ(\e)_{\nu_1}^\circ\cdots(\e)_{\nu_n}^\circ}\prod_{i=1}^n \nu_i(\lambda_i).
\end{align*}
In particular, Appell's functions are the case when $n=2$ and we write $F_1, F_2, F_3$ and $F_4$ for $F_D^{(2)}, F_A^{(2)}, F_B^{(2)}$ and $F_C^{(2)}$, respectively.

\begin{rem}
A priori, the hypergeometric functions over $\k$ are $\Q(\mu_{p(q-1)})$-valued.
However, in fact, Lauricella's functions and ${}_{n+1}F_n$-functions over $\k$ take values in $\Q(\mu_{q-1})$, and they are independent of the choice of the additive character $\psi$ (see \cite[Lemma 2.4 (iii)]{Otsubo}).
\end{rem}

\begin{rem}\label{FA=FB} By (\ref{Poch formula}), (\ref{inversion Poch}) and (\ref{Gauss sum thm}), one shows for $\lambda_i\in\k^\times$,
\begin{align*}
& \FB{n}{\a_1,\dots,\a_n\, ;\, \b_1,\dots,\b_n}{\c}{\lambda_1,\dots,\lambda_n}\\
& = (\ol{\c})_{\b_1\cdots\b_n}\Big(\prod_{i=1}^n (\a_i)_{\ol{\b_i}}\ol{\b_i}(\lambda_i)\Big) \FA{n}{\b_1\cdots\b_n\ol{\c}\, ;\, \b_1,\dots,\b_n}{\ol{\a_1}\b_1,\dots,\ol{\a_n}\b_n}{\dfrac{1}{\lambda_1},\dots,\dfrac{1}{\lambda_n}}.
\end{align*}
\end{rem}

For a function $f:(\k^\times)^n\rightarrow \ol{\Q}$, its {\it Fourier transform} is a function on $(\widehat{\k^\times})^n$ defined by
\begin{equation*}
\widehat{f}(\nu_1,\dots,\nu_n)=\sum_{t_i\in\k^\times}f(t_1,\dots,t_n)\prod_{i=1}^n\ol{\nu_i}(t_i).
\end{equation*}
Then, 
\begin{equation}\label{Fourier trans.}
f(\lambda_1,\dots,\lambda_n)=\dfrac{1}{(q-1)^n}\sum_{\nu_i\in\widehat{\k^\times}}\widehat{f}(\nu_1,\dots,\nu_n)\prod_{i=1}^n \nu_i(\lambda_i).
\end{equation}

\section{Finite analogues of integral representations and their applications}

\subsection{Finite analogues of integral representations}

Over $\C$, Lauricella's $F_A^{(n)}$ and $F_C^{(n)}$ have integral representations (cf. \cite[p.285]{Srivastava-Karlsson} and \cite{Srivastava-Exton}) 
\begin{align*}
& \Gamma(a) \FA{n}{a\, ;\, b_1, \dots, b_n}{c_1,\dots,c_n}{z_1, \dots, z_n}
 = \int_0^\infty e^{-t} t^{a-1} \prod_{i=1}^n \hF{1}{1}{b_i}{c_i}{z_i t}\, dt
\end{align*}
and
$$\Gamma(2a) \FC{n}{a\, ;\, a+\frac{1}{2}}{b_1,\dots,b_n}{z_1,\dots, z_n} = \int_0^\infty e^{-t}t^{2a-1} \prod_{i=1}^n \hF{0}{1}{}{b_i}{\dfrac{z_i t^2}{4}}\, dt.$$
We obtain analogous formulas as the following.
\begin{thm}\label{FA and FC integral rep.}\,
\begin{enumerate}
\item We have, for $\l_i \in \k$,
\begin{align*}
 -g(\a) \FA{n}{\a\, ;\, \b_1,\dots,\b_n}{\c_1,\dots,\c_n}{\l_1,\dots,\l_n} = \sum_{t\in\k^\times} \psi(t) \a(t) \prod_{i=1}^n \hF{1}{1}{\b_i}{\c_i}{\l_i t}.
\end{align*}

\item For any $\l_i \in \k$,
\begin{align*}
& -\dfrac{q}{g^\circ(\c)} \FB{n}{\a_1,\dots,\a_n\, ;\, \b_1,\dots,\b_n}{\c}{\l_1,\dots,\l_n}\\
& = \sum_{t\in\k^\times} \psi(-t) \ol{\c}(t) \prod_{i=1}^n \hF{2}{0}{\a_i, \b_i}{}{\dfrac{\l_i}{t}}.
\end{align*}

\item Suppose that $p \neq 2$. 
For $\l_i \in \k$, we have
\begin{align*}
- g(\a^2) \FC{n}{\a\, ;\, \a\phi}{\b_1,\dots, \b_n}{\l_1, \dots, \l_n} = \sum_{t \in \k^\times} \psi(t) \a^2(t) \prod_{i=1}^n \hF{0}{1}{}{\b_i}{\dfrac{\l_i t^2}{4}}.
\end{align*}

\item For any $\l_i \in \k$,
\begin{align*}
& g(\a)g(\b) \FC{n}{\a\, ;\, \b}{\c_1,\dots,\c_n}{\l_1,\dots,\l_n} \\
& = \sum_{s, t \in \k^\times} \psi(s+t)\a(s)\b(t) \prod_{i=1}^n \hF{0}{1}{}{\c_i}{\l_i st}.
\end{align*}
\end{enumerate}
\end{thm}

\begin{proof}
(i) The right-hand side of the formula is equal to
\begin{align*}
& \dfrac{1}{(1-q)^n} \sum_{\nu_1, \dots, \nu_n} \prod_{i} \dfrac{(\b_i)_{\nu_i}}{(\e)_{\nu_i}^\circ (\c_i)_{\nu_i}^\circ}\nu_i(\l_i) \sum_t \psi(t) \a\nu_1\cdots\nu_n(t)\\
& = -\dfrac{1}{(1-q)^n} \sum_{\nu_1, \dots, \nu_n} g(\a\nu_1\cdots\nu_n) \prod_{i} \dfrac{(\b_i)_{\nu_i}}{(\e)_{\nu_i}^\circ (\c_i)_{\nu_i}^\circ}\nu_i(\l_i) \\
& = -\dfrac{g(\a)}{(1-q)^n} \sum_{\nu_1, \dots, \nu_n} (\a)_{\nu_1\cdots\nu_n} \prod_{i} \dfrac{(\b_i)_{\nu_i}}{(\e)_{\nu_i}^\circ (\c_i)_{\nu_i}^\circ}\nu_i(\l_i).
\end{align*}
Thus, we obtain (i).
The proof of (ii) and (iv) follows similarly. 

(iii) Similarly to (i), the right-hand side of the formula is equal to
\begin{align*}
& -\dfrac{g(\a^2)}{(1-q)^n} \sum_{\nu_1, \dots,\nu_n} (\a^2)_{\nu_1^2\cdots\nu_n^2} \prod_i \dfrac{1}{(\e)_{\nu_i}^\circ (\b_i)_{\nu_i}^\circ} \nu_i\Big(\dfrac{\l_i}{4}\Big).
\end{align*}
By \eqref{dup. form. Poch}, we have
$$(\a^2)_{\nu_1^2\cdots\nu_n^2} = \nu_1\cdots\nu_n(4) (\a)_{\nu_1\cdots\nu_n}(\a\phi)_{\nu_1\cdots\nu_n},$$
and hence we obtain (iii).
\end{proof}

%

\begin{rem}\label{rem1}
Also, Lauricella's $F_D$ over $\C$ has the similar integral representation (cf. \cite[Theorem 3.4.1]{Matsumoto})
\begin{align*}
& B(a, c-a) \FD{n}{a\, ;\, b_1, \dots, b_n}{c}{z_1,\dots,z_n} \\
& = \int_0^1 t^{a-1} (1-t)^{c-a-1} \prod_i (1-z_i t)^{-b_i} \, dt\\
& = \int_0^1 t^{a-1} (1-t)^{c-a-1} \prod_i \hF{1}{0}{b_i}{}{z_i t} \, dt.
\end{align*}
Here, note that $\hF{1}{0}{a}{}{z} = (1-z)^{-a}$.
A finite field analogue of this integral representation is obtained by the author \cite[Theorem 3.1 (i)]{N}.
\end{rem}

\begin{rem}
Lauricella's functions over $\C$ have another type of integral representations of Euler type (cf \cite[Theorem 3.4.1]{Matsumoto}).
The author \cite{N} obtains their finite field analogues, and expresses the numbers of rational points on certain algebraic varieties over a finite field in terms of Lauricella's functions over the finite field.
\end{rem}

\subsection{Formulas for confluent hypergeometric functions over $\k$}
Due to Otsubo \cite[Proposition 2.9 (ii)]{Otsubo}, we have
\begin{equation}\label{0F0=psi}
\hF{0}{0}{}{}{\l} = \psi(-\l)\quad (\l \in \k^\times).
\end{equation}
When $\a_i \neq \b_i$ ($i = 1, \dots, n$) and $\l \neq 0$, we also have (\cite[Theorem 3.3]{Otsubo})
\begin{align}
& (-1)^n\Big( \prod_{i=1}^n j(\a_i, \ol{\a_i}\b_i) \Big) \hF{n}{n}{\a_1,\dots,\a_n}{\b_1,\dots,\b_n}{\l} \label{1F1 int.}\\
& = \sum_{u_i \in \k} \psi(-\l u_1 \cdots u_n) \prod_{i=1}^n \a_i(u_i) \ol{\a_i}\b_i(1-u_i). \nonumber
\end{align}
Recall the following formula due to Erd\'elyi \cite[below (18)]{Erdelyi}:
\begin{align*}
& \sum_{n=0}^\infty \dfrac{(a_1)_n (a_2)_n}{(1)_n(b_1)_n (b_2)_n} \hF{1}{1}{a_1+n}{b_1+n}{-z} \hF{1}{1}{a_2+n}{b_2+n}{-z}z^n\\
& = e^{-z} \hF{2}{2}{b_1-a_1, b_2-a_2}{b_1, b_2}{z}.
\end{align*}
An analogue of this formula is the following.
\begin{thm}
Suppose that $\a_i \not\in\{\e,\b_i\}$ $(i=1, 2)$.
Then, for $\l\in\k$,
\begin{align*}
& \dfrac{1}{1-q} \sum_{\nu\in\khat} \dfrac{(\a_1)_\nu (\a_2)_\nu}{(\e)_\nu^\circ(\b_1)_\nu^\circ (\b_2)_\nu^\circ}\hF{1}{1}{\a_1\nu}{\b_1\nu}{-\l}\hF{1}{1}{\a_2\nu}{\b_2\nu}{-\l}\nu(\l)\\
& = \psi(\l) \hF{2}{2}{\ol{\a_1}\b_1, \ol{\a_2}\b_2}{\b_1, \b_2}{\l}.
\end{align*}
\end{thm}
\begin{proof}
When $\l = 0$, the theorem is clear.
Hence, we may assume that $\l \neq 0$.
By \eqref{1F1 int.} and \eqref{J=G}, we have, for each $\nu$,
\begin{align*}
& \dfrac{(\a_1)_\nu (\a_2)_\nu}{(\b_1)_\nu^\circ (\b_2)_\nu^\circ}\hF{1}{1}{\a_1\nu}{\b_1\nu}{-\l}\hF{1}{1}{\a_2\nu}{\b_2\nu}{-\l}\\
& = j(\a_1, \ol{\a_1}\b_1)^{-1}j(\a_2, \ol{\a_2}\b_2)^{-1} \\
&\quad \times \sum_{u, v} \psi(u\l+v\l) \a_1(u)\ol{\a_1}\b_1(1-u) \a_2(v)\ol{\a_2}\b_2(1-v) \nu(uv).
\end{align*}
Thus, the left-hand side of the theorem is equal to
\begin{align*}
& j(\a_1, \ol{\a_1}\b_1)^{-1}j(\a_2, \ol{\a_2}\b_2)^{-1} \\
& \quad \times \sum_{u, v} \psi(u\l+v\l) \a_1(u)\ol{\a_1}\b_1(1-u) \a_2(v)\ol{\a_2}\b_2(1-v) \hF{0}{0}{}{}{uv\l}\\
& = j(\a_1, \ol{\a_1}\b_1)^{-1}j(\a_2, \ol{\a_2}\b_2)^{-1} \\
& \quad \times \sum_{u, v} \psi(u\l+v\l-uv\l) \a_1(u)\ol{\a_1}\b_1(1-u) \a_2(v)\ol{\a_2}\b_2(1-v).
\end{align*}
Here, we used \eqref{0F0=psi}. 
Noting that $\psi(u\l +v\l -uv\l) = \psi(\l)\psi(-\l(1-u)(1-v))$ and putting $z = 1-u$ and $w = 1-v$, the right-hand side above is equal to
\begin{align*}
& \psi(\l) j(\a_1, \ol{\a_1}\b_1)^{-1}j(\a_2, \ol{\a_2}\b_2)^{-1} \\
& \quad \times \sum_{z,w} \psi(-\l zw) \ol{\a_1}\b_1(z)\a_1(1-z) \ol{\a_2}\b_2(w)\a_2(1-w).
\end{align*}
Thus, noting the assumption $\a_i \neq \e$, we obtain the theorem by \eqref{1F1 int.}.
\end{proof}

We will use the following formulas in the next subsections.
\begin{prop}\label{lem 1F1}
Suppose that $\nu \neq \e$. Then,
\begin{align*}
& \dfrac{1}{1-q} \sum_{\eta\in\khat} \hF{1}{1}{\ol{\eta}}{\e}{x} \hF{1}{1}{\ol{\nu}\eta}{\e}{y} \\
& = \hF{1}{1}{\ol{\nu}}{\e}{x+y} - \psi(-y) \hF{1}{1}{\ol{\nu}}{\e}{x} - \psi(-x) \hF{1}{1}{\ol{\nu}}{\e}{y}.
\end{align*}
\end{prop}
\begin{proof}
When $x=0$ or $y=0$, the proposition is clear because the both sides become 0, and hence we may assume that $x, y\neq 0$.
One shows 
\begin{equation}
\hF{1}{1}{\a}{\e}{\l} = - \a(-1)\sum_{u \neq 0,1} \psi(-\l u) \a\Big(\dfrac{u}{1-u}\Big) + \d(\a) (q-1)\psi(- \l). \label{1F1 int. ana.}
\end{equation}
Indeed, it is derived by \eqref{1F1 int.} when $\a \neq \e$, and the both sides are equal to $1+q\psi(-\l)$ when $\a=\e$ by \eqref{0F0=psi} and $-\sum_{u \neq 0,1} \psi(-\l u) = \psi(0) + \psi(-\l) = 1 + \psi(-\l)$.

By \eqref{1F1 int. ana.}, 
\begin{align*}
& \dfrac{1}{1-q}\sum_\eta \hF{1}{1}{\ol\eta}{\e}{x} \hF{1}{1}{\ol\nu \eta}{\e}{y}\\
& = \dfrac{\nu(-1)}{1-q}\sum_{u,v \neq 1} \psi(-xu -yv) \ol{\nu}\Big( \dfrac{v}{1-v}\Big) \sum_\eta \eta\Big(\dfrac{v(1-u)}{u(1-v)}\Big)\\
& \quad + \nu(-1)\psi(-y) \sum_{u\neq 1} \psi(-xu)\ol\nu \Big( \dfrac{u}{1-u}\Big) + \nu(-1)\psi(-x) \sum_{v\neq 1} \psi(-yv)\ol\nu\Big(\dfrac{v}{1-v}\Big).
\end{align*}
Noting that, by the orthogonality of characters,
$$ \sum_\eta \eta\Big(\dfrac{v(1-u)}{u(1-v)}\Big) = 
\begin{cases}
0 & (u \neq v),\\
q-1 & (u=v),
\end{cases}
$$
the first term of the right-hand side above is equal to
\begin{align*}
-\nu(-1)\sum_{u\neq 1} \psi( -(x+y)u) \ol\nu\Big( \dfrac{u}{1-u}\Big).
\end{align*}
Thus, we obtain the proposition by \eqref{1F1 int. ana.}.
\end{proof}
The following is a finite analogue of Kummer's product formula.
\begin{prop}[{\cite[Theorem 6.1 (ii)]{Otsubo}}]\label{Kummer second}
Suppose that $p \neq 2$.
If $\a \neq \e$, then for $\l \in\k$,
$$ \psi(\l) \hF{1}{1}{\a}{\a^2}{2\l} = \hF{0}{1}{}{\a\phi}{\dfrac{\l^2}{4}}.$$
\end{prop}

\subsection{Reduction, summation and transformation formulas for $F_A$ and $F_C$}

Over $\C$, Padmanabham-Srivastava \cite[(7)]{Padmanabham-Srivastava} show that
\begin{align*}
& \sum_{k=0}^n \FA{n+2}{a\, ;\, -k, -n+k, b_1, \dots, b_n}{1, 1, c_1,\dots, c_n}{x, y, z_1, \dots, z_n}\\
& = (n+1) \FA{n+1}{a\, ;\, -n, b_1, \dots, b_n}{2, c_1, \dots, c_n}{x+y, z_1, \dots, z_n},
\end{align*}
and 
$$ \sum_{k=0}^n \FAA{ 1\, ;\, -k, n-k }{ 1, 1}{\dfrac{1}{2}, \dfrac{1}{2}} = 1.$$
We obtain finite field analogues of them below.

\begin{prop}\label{red. of FA}
Suppose that $\nu\neq\e$. Then, for $x,y \in \k-\{1\}$ and $\l_1,\dots,\l_n\in\k$,
\begin{align*}
& \dfrac{1}{1-q} \sum_{\eta\in\khat} \FA{n+2}{\a\, ;\, \ol{\eta}, \ol{\nu}\eta, \b_1,\dots,\b_n}{\e,\e, \c_1,\dots,\c_n}{x,y,\l_1,\dots,\l_n}\\
& = \FA{n+1}{\a\, ;\, \ol{\nu},\b_1,\dots,\b_n}{\e,\c_1,\dots,\c_n}{x+y,\l_1,\dots,\l_n}\\
& \quad - \ol{\a}(1-x) \FA{n+1}{\a\, ;\, \ol{\nu}, \b_1,\dots,\b_n}{\e,\c_i,\dots,\c_n}{\dfrac{y}{1-x},\dfrac{\l_1}{1-x},\dots,\dfrac{\l_n}{1-x}}\\
& \quad - \ol{\a}(1-y) \FA{n+1}{\a\, ;\, \ol{\nu}, \b_1,\dots,\b_n}{\e,\c_i,\dots,\c_n}{\dfrac{x}{1-y},\dfrac{\l_1}{1-y},\dots,\dfrac{\l_n}{1-y}}.
\end{align*}
\end{prop}
\begin{proof}
By Theorem \ref{FA and FC integral rep.} (i), we have
\begin{align*}
& \dfrac{-g(\a)}{1-q} \sum_{\eta\in\khat} \FA{n+2}{\a\, ;\, \ol{\eta}, \ol{\nu}\eta, \b_1,\dots,\b_n}{\e,\e, \c_1,\dots,\c_n}{x,y,\l_1,\dots,\l_n}\\
& = \dfrac{1}{1-q}  \sum_t \psi(t) \a(t) \sum_\eta \hF{1}{1}{\ol{\eta}}{\e}{x t} \hF{1}{1}{\ol{\nu}\eta}{\e}{y t} \prod_i \hF{1}{1}{\b_i}{\c_i}{\l_i t} .
\end{align*}
Using Proposition \ref{lem 1F1} and putting $s = t - xt$ or $s = t-yt$, the right-hand side above is equal to
\begin{align*}
& \sum_t \psi(t)\a(t) \hF{1}{1}{\ol{\nu}}{\e}{(x+y)t}\prod_i \hF{1}{1}{\b_i}{\c_i}{\l_i t}\\
& \quad - \sum_t \psi(t-xt)\a(t) \hF{1}{1}{\ol{\nu}}{\e}{yt}\prod_i \hF{1}{1}{\b_i}{\c_i}{\l_i t}\\
& \quad - \sum_t \psi(t-yt)\a(t) \hF{1}{1}{\ol{\nu}}{\e}{xt}\prod_i \hF{1}{1}{\b_i}{\c_i}{\l_i t}\\
& = \sum_t \psi(t)\a(t) \hF{1}{1}{\ol{\nu}}{\e}{(x+y)t}\prod_i \hF{1}{1}{\b_i}{\c_i}{\l_i t}\\
& \quad - \ol{\a}(1-x)\sum_s \psi(s)\a(s) \hF{1}{1}{\ol{\nu}}{\e}{\frac{ys}{1-x}}\prod_i \hF{1}{1}{\b_i}{\c_i}{\frac{\l_i s}{1-x}}\\
& \quad - \ol{\a}(1-y)\sum_s \psi(s)\a(s) \hF{1}{1}{\ol{\nu}}{\e}{\frac{xs}{1-y}}\prod_i \hF{1}{1}{\b_i}{\c_i}{\frac{\l_i s}{1-y}}.\\
\end{align*}
Thus, we obtain the proposition by Theorem \ref{FA and FC integral rep.} (i).
\end{proof}

By the proposition above with $n=0$, we have
\begin{align*}
& \dfrac{1}{1-q}\sum_\eta \FAA{\a\, ;\, \ol{\eta}, \ol{\nu}\eta}{\e,\e}{x,y}\\
& = \hF{2}{1}{\a,\ol{\nu}}{\e}{x+y} - \ol{\a}(1-x) \hF{2}{1}{\a,\ol{\nu}}{\e}{\dfrac{y}{1-x}} - \ol{\a}(1-y)\hF{2}{1}{\a,\ol{\nu}}{\e}{\dfrac{x}{1-y}}.
\end{align*}
Hereafter, we suppose that $p \neq 2$.
Taking $x=y=\tfrac{1}{2}$ and using the formula above and \eqref{Euler-Gauss}, we have
\begin{align*}
 \dfrac{1}{1-q}\sum_\eta \FAA{\a\, ;\, \ol{\eta}, \ol{\nu}\eta}{\e, \e}{\dfrac{1}{2},\dfrac{1}{2}}
& = (1-2\a(2)) \hF{2}{1}{\a,\ol{\nu}}{\e}{1}\\
& = (1-2\a(2)) \dfrac{q g(\ol{\a}\nu)}{g^\circ(\ol{\a}) g^\circ(\nu)}.
 \end{align*}
Thus, taking $\a = \e$, we obtain the following.
\begin{cor}
Suppose that $p \neq 2$ and $\nu \neq \e$. 
Then,
$$ \dfrac{1}{1-q}\sum_{\eta\in\khat} \FAA{\e\, ;\, \ol\eta, \ol\nu \eta}{\e, \e}{\dfrac{1}{2}, \dfrac{1}{2}} = -1.$$
\end{cor}

Recall the Bailey's formula (cf. \cite[(4.7) and (4.8)]{Bailey})
\begin{align*}
& \FC{n}{a\, ;\, a+\frac{1}{2}}{b_1+\frac{1}{2}, \dots,b_n+\frac{1}{2}}{z_1^2,\dots,z_n^2}\\
& =\Big( 1+\sum_{i=1}^n z_i \Big)^{-2a} \FA{n}{2a\, ;\, b_1,\dots,b_n}{2b_1,\dots,2b_n}{\dfrac{2z_1}{1+\sum_{i=1}^n z_i},\dots,\dfrac{2z_n}{1+\sum_{i=1}^n z_i}},
\end{align*} 
(Srivastava-Exton \cite{Srivastava-Exton} give an alternative proof).
The following is a finite field analogue of this formula.

\begin{thm}\label{FA = FC}
Suppose that $p \neq 2$ and $\b_i \neq \e$ for all $i = 1,\dots n$.
For $\l_i \in \k$ with $\sum_{i=1}^n \l_i \neq -1$, we have
\begin{align*}
&\FC{n}{\a \, ;\, \a\phi}{\b_1\phi, \dots, \b_n\phi}{\l_1^2, \dots, \l_n^2}\\
& = \ol{\a}^2 \Big( 1 + \sum_{i=1}^n \l_i \Big) \FA{n}{\a^2\, ;\, \b_1 , \dots, \b_n}{\b_1^2, \dots, \b_n^2}{\dfrac{2\l_1}{1+\sum_{i=1}^n \l_i},\dots, \dfrac{2\l_n}{1+\sum_{i=1}^n \l_i}}.
\end{align*}
\end{thm}

\begin{proof}
By Proposition \ref{Kummer second}, we have
$$ \prod_i \hF{0}{1}{}{\b_i\phi}{\dfrac{\l_i^2}{4}} = \psi\Big(\sum_i \l_i\Big) \prod_i \hF{1}{1}{\b_i}{\b_i^2}{2\l_i}.$$
By Theorem \ref{FA and FC integral rep.} (iii) and the equation above, we have
\begin{align*}
& -g(\a^2) \FC{n}{\a\, ;\, \a\phi}{\b_1\phi, \dots, \b_n\phi}{\l_1^2,\dots,\l_n^2}\\
& = \sum_t \psi(t) \a^2(t) \prod_i \hF{0}{1}{}{\b_i\phi}{\dfrac{\l_i^2 t^2}{4}}\\
& = \sum_t \psi \Big(t + \sum_i \l_i t \Big) \a^2(t) \prod_i \hF{1}{1}{\b_i}{\b_i^2}{2\l_i t}\\
& = \ol{\a}^2\Big(1 + \sum_i \l_i \Big) \sum_s \psi(s) \a^2(s) \prod_i \hF{1}{1}{\b_i}{\b_i^2}{\dfrac{2\l_i s}{1+\sum_i \l_i}}.
\end{align*}
Here, we putted $s = t(1+\sum_i \l_i)$.
We obtain the theorem by Theorem \ref{FA and FC integral rep.} (i).
\end{proof}

\begin{rem}
For the case when $n=2$ and $\b_i^2 \neq \e$, the formula in Theorem \ref{FA = FC} is essentially proved by Kalita-Azharuddin \cite[Theorem 1.4]{Kalita-Azharuddin}.
\end{rem}

\subsection{Reduction formulas for Appell's functions $F_2$ over $\k$}

Over $\C$, the following formulas are known (cf. \cite[(92) and (96)]{Vidunas}):
$$\hF{2}{1}{a, a+\frac{1}{2}}{\frac{1}{2}}{z^2} = \dfrac{(1+z)^{-2a} + (1 - z)^{-2a}}{2},$$
and
\begin{align*}
\dfrac{(a+\frac{1}{2})_l}{(\frac{1}{2})_l} \hF{2}{1}{a, a+l+\frac{1}{2}}{\frac{1}{2}-k}{z^2} 
= & \dfrac{(1+z)^{-2a}}{2} \FAA{2a\, ;\, -k, -l}{-2k, -2l}{\dfrac{2 z}{1+z}, \dfrac{2}{1+z}}\\
& + \dfrac{(1-z)^{-2a}}{2} \FAA{2a\, ;\, -k, -l}{-2k, -2l}{-\dfrac{2 z}{1-z}, \dfrac{2}{1-z}},
\end{align*}
where $k$ and $l$ are non-negative integers.
The following theorem is finite field analogues of these formulas. 
In fact, (i) of the following theorem is essentially given in \cite[Theorem 8.11]{FLRST}.
\begin{thm}Suppose that $p \neq 2$.\,
\begin{enumerate}
\item If $\a^2\neq \e$ or $\l \neq \pm 1$, then
$$\hF{2}{1}{\a, \a\phi}{\phi}{\l^2} = \dfrac{\ol{\a}^2(1+\l) + \ol{\a}^2(1-\l)}{2}\quad\quad (\l \neq 0).$$

\item Suppose that $\e \not \in \{\chi, \eta\}$. 
For $\l \neq -1$, 
\begin{align*}
&\sum_{\a' \in \{\a, \a\phi\}} \dfrac{(\eta\phi)_{\a'}}{(\phi)_{\a'}} \hF{2}{1}{\a' , \a'\eta\phi}{\ol{\chi}\phi}{\l^2} = \ol{\a}^2 (1+\l) \FAA{\a^2\, ;\, \ol{\chi}, \ol{\eta}}{\ol{\chi}^2, \ol{\eta}^2}{\dfrac{2\l}{1+\l}, \dfrac{2}{1+\l}}.
\end{align*}
\end{enumerate}
\end{thm}


\begin{proof}
(i) Note that
$$\sum_{\nu\in\khat} F(\nu^2)\nu^2(\l) = \dfrac{1}{2}\sum_{\nu\in\khat} (1+\nu(-1))F(\nu)\nu(\l) = \dfrac{1}{2}\Big( \sum_{\nu\in\khat} F(\nu)\nu(\l) + \sum_{\nu\in\khat} F(\nu)\nu(-\l) \Big),$$
for any $F:\khat \rightarrow \ol\Q^\times$ and $\l \in \k$.
By \eqref{dup. form. Poch},
\begin{align*}
\hF{2}{1}{\a, \a\phi}{\phi}{\l^2}
& = \dfrac{1}{1-q} \sum_\nu \dfrac{(\a)_\nu (\a\phi)_\nu}{(\e)_\nu^\circ (\phi)_\nu^\circ}\nu^2(\l)\\
& = \dfrac{1}{1-q} \sum_\nu \dfrac{(\a^2)_{\nu^2}}{(\e)_{\nu^2}^\circ}\nu^2(\l)\\
& = \dfrac{1}{2(1-q)} \Big( \sum_\nu \dfrac{(\a^2)_{\nu}}{(\e)_{\nu}^\circ}\nu(\l) + \sum_\nu \dfrac{(\a^2)_{\nu}}{(\e)_{\nu}^\circ}\nu(-\l) \Big)\\
& = \dfrac{1}{2}\Big( \hF{1}{0}{\a^2}{}{\l} + \hF{1}{0}{\a^2}{}{-\l} \Big).
\end{align*}
Thus, we obtain (i) by \eqref{1F0 int.}.

(ii) Put
$$ f(\l) := \sum_s \a^2(s) \hF{0}{1}{}{\ol{\chi}\phi}{\dfrac{\l s^2}{4}}\hF{0}{1}{}{\ol{\eta}\phi}{\dfrac{s^2}{4}}.$$
By Theorem \ref{FA and FC integral rep.} (i), putting $s = t / (1+\l)$ and Proposition \ref{Kummer second}, we have
\begin{align*}
& \ol{\a}^2(1+\l) \FAA{\a^2 \, ;\, \ol{\chi}, \ol{\eta}}{\ol{\chi}^2, \ol{\eta}^2}{\dfrac{2\l}{1+\l}, \dfrac{2}{1+\l}}\\
& = - \dfrac{\ol{\a}^2(1+\l)}{g(\a^2)} \sum_t \psi(t) \a^2(t) \hF{1}{1}{\ol{\chi}}{\ol{\chi}^2}{\dfrac{2\l t}{1+\l}} \hF{1}{1}{\ol{\eta}}{\ol{\eta}^2}{\dfrac{2t}{1+\l}}\\
& = - \dfrac{1}{g(\a^2)} \sum_s \a^2(s) \psi(\l s) \hF{1}{1}{\ol{\chi}}{\ol{\chi}^2}{2\l s} \psi(s)\hF{1}{1}{\ol{\eta}}{\ol{\eta}^2}{2s}\\
& = -\dfrac{1}{g(\a^2)} f(\l^2).
\end{align*}
Using \eqref{inversion Poch} and \eqref{Poch formula}, we have, for any $\mu \in \khat$, 
\begin{align*}
\widehat{f}(\mu) 
& = \sum_t f(t) \ol{\mu}(t)\\
& = \dfrac{1}{(1-q)^2} \sum_s \a^2(s) \sum_{\nu_1, \nu_2} \dfrac{1}{(\e)_{\nu_1}^\circ (\ol{\chi}\phi)_{\nu_1}^\circ (\e)_{\nu_2}^\circ (\ol{\eta}\phi)_{\nu_2}^\circ}\nu_1\nu_2\Big( \dfrac{s^2}{4} \Big) \sum_t \nu_1\ol{\mu}(t)\\
& = \dfrac{1}{q-1}\cdot \dfrac{1}{(\e)_\mu^\circ (\ol{\chi}\phi)_\mu^\circ} \sum_{\nu_2} \dfrac{\ol{\mu\nu_2}(4)}{(\e)_{\nu_2}^\circ (\ol{\eta}\phi)_{\nu_2}^\circ} \sum_s \a^2\mu^2\nu_2^2(s)\\
& = \dfrac{1}{(\e)_\mu^\circ (\ol{\chi}\phi)_\mu^\circ} \sum_{\a'\in\{\a,\a\phi\}} \dfrac{\a(4)}{(\e)_{\ol{\a'\mu}}^\circ (\ol{\eta}\phi)_{\ol{\a'\mu}}^\circ}\\
& = \dfrac{\a(4)}{(\e)_\mu^\circ (\ol{\chi}\phi)_\mu^\circ} \sum_{\a'\in\{\a,\a\phi\}} (\e)_{\a'\mu} (\eta\phi)_{\a'\mu}\\
& = \a(4) \sum_{\a'\in\{\a,\a\phi\}} (\e)_{\a'} (\eta\phi)_{\a'} \dfrac{(\a')_\mu (\a'\eta\phi)_\mu}{(\e)_\mu^\circ (\ol{\chi}\phi)_\mu^\circ}.
\end{align*}
Thus, by \eqref{Fourier trans.}, 
\begin{align*}
f(\l^2)
& = -\a(4)\sum_{\a'\in\{\a,\a\phi\}} (\e)_{\a'} (\eta\phi)_{\a'} \hF{2}{1}{\a', \a'\eta\phi}{\ol{\chi}\phi}{\l^2},
\end{align*}
and hence,
\begin{align*}
& \ol{\a}^2(1+\l) \FAA{\a^2 ; \ol{\chi}, \ol{\eta}}{\ol{\chi}^2, \ol{\eta}^2}{\dfrac{2\l}{1+\l}, \dfrac{2}{1+\l}}\\
& = \dfrac{\a(4)}{g(\a^2)}\sum_{\a'\in\{\a,\a\phi\}} (\e)_{\a'} (\eta\phi)_{\a'} \hF{2}{1}{\a', \a'\eta\phi}{\ol{\chi}\phi}{\l^2}.
\end{align*}
By \eqref{dup. form. gauss}, we have, for each $\a' \in \{\a, \a\phi\}$, 
\begin{align}
g(\a^2) = \a(4) \dfrac{g(\a') g(\a'\phi)}{g(\phi)} = \a(4) (\e)_{\a'} (\phi)_{\a'},\label{dup}
\end{align}
and hence, we obtain the theorem.
\end{proof}

The following two functions satisfy the same second order ordinary differential equation (\cite[Theorem 2.4]{Vidunas}):
$$\FAA{2a\, ;\, b_1, b_2}{2b_1, 2b_2}{1+z, 1-z}\quad \mathrm{and}\quad (1-z)^{-2a} \hF{2}{1}{a, a-b_2+\frac{1}{2}}{b_1+\frac{1}{2}}{\Big(\dfrac{1+z}{1-z}\Big)^2}.$$
Over $\k$, we have the following relation between their corresponding functions.
\begin{thm}\label{F2 = 2F1}
Suppose that $p \neq 2$ and $\e\not\in\{\b_1,\b_2\}$.
Then, for $\l \neq 1$,
\begin{align*}
& \FAA{\a^2\, ;\, \b_1, \b_2}{\b_1^2, \b_2^2}{1+\l, 1-\l} \\
& = \ol{\a}^2\Big(\dfrac{1-\l}{2}\Big) \sum_{\a' \in \{\a, \a\phi\}} \dfrac{(\ol{\b_2}\phi)_{\a'}}{(\phi)_{\a'}}\hF{2}{1}{\a', \a'\ol{\b_2}\phi}{\b_1\phi}{\Big(\dfrac{1+\l}{1-\l}\Big)^2}. 
\end{align*}
\end{thm}

\begin{proof}
By Theorem \ref{FA and FC integral rep.} (i) and Proposition \ref{Kummer second}, we have
\begin{align*}
& g(\a^2) \FAA{\a^2\, ;\, \b_1, \b_2}{\b_1^2, \b_2^2}{1+\l, 1-\l}\\
& = -\sum_t \psi(t) \a^2(t) \hF{1}{1}{\b_1}{\b_1^2}{(1+\l)t} \hF{1}{1}{\b_2}{\b_2^2}{(1-\l) t}\\
& = -\sum_t \a^2(t) \hF{0}{1}{}{\b_1\phi}{\dfrac{(1+\l)^2 t^2}{16}} \hF{0}{1}{}{\b_2\phi}{\dfrac{(1-\l)^2t^2}{16}}\\
& = -\dfrac{1}{(1-q)^2}\sum_{\nu, \mu} \dfrac{1}{(\e)_\nu^\circ (\b_1\phi)_\nu^\circ (\e)_\mu^\circ (\b_2\phi)_\mu^\circ}\nu^2\Big( \dfrac{1+\l}{4}\Big) \mu^2\Big( \dfrac{1-\l}{4}\Big) \sum_t \a^2\nu^2\mu^2(t)\\
& = \ol{\a}^2\Big(\dfrac{1-\l}{4}\Big) \dfrac{1}{1-q} \sum_{\a' \in \{\a, \a\phi\}} \sum_\nu \dfrac{1}{(\e)_\nu^\circ (\b_1\phi)_\nu^\circ (\e)_{\ol{\a'\nu}}^\circ (\b_2\phi)_{\ol{\a'\nu}}^\circ} \nu^2\Big( \dfrac{1+\l}{1-\l}\Big).
\end{align*}
Using \eqref{inversion Poch} and \eqref{Poch formula}, the last right-hand side above is equal to
\begin{align*}
& \ol{\a}^2\Big( \dfrac{1-\l}{4}\Big) \sum_{\a' \in \{\a, \a\phi\}} (\e)_{\a'} (\ol{\b_2}\phi)_{\a'} \hF{2}{1}{\a', \a'\ol{\b_2}\phi}{\b_1\phi}{\Big(\dfrac{1+\l}{1-\l}\Big)^2}.
\end{align*}
By \eqref{dup}, we obtain the theorem.
\end{proof}

Vid\=unas \cite[Theorem 2.1]{Vidunas} also proves the following.
\begin{enumerate}
\item The following functions satisfy the same ordinary differential equation of order 3:
$$\FAA{a\, ;\, b_1, b_2}{c_1, c_2}{z, 1} \quad \mathrm{and} \quad \hF{3}{2}{a, b_1, a-c_2+1}{c_1, a+b_2-c_2+1}{z}.$$

\item The following functions satisfy the same ordinary differential equation of order 3:
$$\FAA{a\, ;\, b_1, b_2}{c_1, c_2}{z, 1-z} \quad \mathrm{and} \quad (1-z)^{-a}\hF{3}{2}{a, c_1-b_1, a-c_2+1}{c_1, a+b_2-c_2+1}{\dfrac{z}{z-1}}.$$
\end{enumerate}
The following theorem provides relations between the corresponding functions over $\k$.
The relations are essentially proved by He-Li-Zhang \cite[Theorems 1.4 and 3.5]{HLZ}.
\begin{thm}\,
\begin{enumerate}
\item Suppose that $\e \not\in \{\a, \b_2, \b_1\ol{\c_1}, \b_2\ol{\c_2}\}$.
Then, for any $\l \in \k$, we have
$$\FAA{\a\, ;\, \b_1,\b_2}{\c_1,\c_1}{\l,1} = \dfrac{(\ol{\b_2}\c_2)_{\ol\a}}{(\c_2)_{\ol\a}^\circ} \hF{3}{2}{\a, \b_1, \a\ol{\c_2}}{\c_1,\a\b_2\ol{\c_2}}{\l}.$$

\item Suppose that $\e \not\in \{\a, \b_1, \b_2, \b_1\ol{\c_1}, \b_2\ol{\c_2}\}$.
Then, for any $\l \in \k-\{1\}$, we have
\begin{align*}
\FAA{\a\, ;\, \b_1,\b_2}{\c_1,\c_2}{\l, 1-\l} = \dfrac{(\ol{\b_2}\c_2)_{\ol\a}}{(\c_2)_{\ol\a}^\circ} \ol{\a}(1-\l) \hF{3}{2}{\a, \ol{\b_1}\c_1, \a\ol{\c_2}}{\c_1, \a\b_2\ol{\c_2}}{\dfrac{\l}{\l-1}}.
\end{align*}
\end{enumerate}
\end{thm}

\begin{proof}
It is trivial when $\l=0$, and hence we may assume that $\l \neq 0$.
By the author \cite[Theorem 3.3]{N}, we have
\begin{align}
&\Big( \prod_{i=1}^2 j(\b_i,\ol{\b_i}\c_i) \Big) \FAA{\a\, ;\, \b_1,\b_2}{\c_1,\c_2}{x,y} \label{FA int.}\\
&=\sum_{u,v\in \k^\times}\ol{\a}(1-xu -yv) \b_1(u)\ol{\b_1}\c_1(1-u) \b_2(v) \ol\b_2\c_2(1-v),\nonumber
\end{align}
when $\e \not\in \{ \a, \ol{\b_1}\c_1, \ol\b_2\c_2 \}$ and $x,y \neq 0$.
By this, we have
\begin{align*}
& \FAA{\a\, ;\, \b_1,\b_2}{\c_1,\c_2}{\l, 1}\\
& = \Big( \prod_{i=1}^2 j(\b_i, \ol{\b_i}\c_i) \Big)^{-1} \sum_{u,v} \ol{\a}(1-\l u - v) \b_1(u) \ol{\b_1}\c_1(1-u) \b_2(v) \ol{\b_2}\c_2(1-v).
\end{align*}
Letting $v = (t-1)/t$ and using \eqref{3F2 int.}, the right-hand side above is equal to
\begin{align*}
& \b_2(-1)\Big( \prod_{i=1}^2 j(\b_i, \ol{\b_i}\c_i) \Big)^{-1}  \sum_{u,t} \ol{\a} (1-\l ut) \b_1(u)\ol{\b_1}\c_1(1-u) \a\ol{\c_2}(t) \b_2(1-t)\\
& = \b_2(-1) j(\b_2, \ol{\b_2}\c_2)^{-1} j(\a\ol{\c_2}, \b_2)  \hF{3}{2}{\a, \b_1, \a\ol{\c_2}}{\c_1,\a\b_2\ol{\c_2}}{\l}.
\end{align*}
By \eqref{J=G} and \eqref{Gauss sum thm}, one shows
\begin{align}
\b_2(-1) j(\b_2, \ol{\b_2}\c_2)^{-1} j(\a\ol{\c_2}, \b_2)  = \dfrac{(\ol{\b_2}\c_2)_{\ol\a}}{(\c_2)_{\ol\a}^\circ},\label{3.13}
\end{align}
and hence, we obtain (i).

By \eqref{FA int.}, we have 
\begin{align*}
& \FAA{\a\, ;\, \b_1,\b_2}{\c_1,\c_2}{\l, 1-\l}\\
& = \Big( \prod_{i=1}^2 j(\b_i, \ol{\b_i}\c_i) \Big)^{-1} \sum_{u,v} \ol{\a}(1-\l u - v(1-\l)) \b_1(u) \ol{\b_1}\c_1(1-u) \b_2(v) \ol{\b_2}\c_2(1-v).
\end{align*}
Putting $s = 1-u$ and $t = 1/(1-v)$ and using \eqref{3F2 int.}, the right-hand side above is equal to
\begin{align*}
& \Big( \prod_{i=1}^2 j(\b_i, \ol{\b_i}\c_i) \Big)^{-1} \ol{\a}(1-\l) \sum_{s,t} \ol{\a}\Big( 1-\dfrac{\l st}{\l-1} \Big) \ol{\b_1}\c_1(s) \b_1(1-s) \a\ol{\c_2}(t)\b_2(1-t)\\
& = \b_2(-1)j(\b_2,\ol{\b_2}\c_2)^{-1} j(\a\ol{\c_2}, \b_2) \ol{\a}(1-\l) \hF{3}{2}{\a, \ol{\b_1}\c_1, \a\ol{\c_2}}{\c_1, \a\b_2\ol{\c_2}}{\dfrac{\l}{\l-1}}.
\end{align*}
Thus, we obtain (ii) by \eqref{3.13}.
\end{proof}

\subsection{Reduction and transformation formulas for Appell's functions $F_3$ over $\k$}

Over $\C$, Vid\=unas \cite[Theorem 3.1]{Vidunas} shows the followings.
\begin{enumerate}
\item The following functions satisfy the same ordinary differential equation of order 3:
$$ \FBB{a_1, a_2\, ;\, b_1,b_2}{c}{z,1} \quad \mathrm{and}\quad \hF{3}{2}{a_1, b_1, c-a_2-b_2}{c-a_2, c-b_2}{z}.$$

\item The following functions satisfy the same ordinary differential equation of order 3:
\begin{align*}
\FBB{a_1, a_2\, ;\, b_1, b_2}{c}{z, \dfrac{z}{z-1}}
\end{align*}
and
\begin{align*}
 z^{1-c}(1-z)^{a_2} \hF{3}{2}{1+a_1+a_2-c, 1+a_2+b_1-c, 1-b_2}{1+a_1+a_2+b_1-c, 1+a_2-b_2}{1-z}.
\end{align*}

\item The following functions satisfy the same ordinary differential equation of order 2:
$$ \FBB{2b_1, 2b_2\, ;\, 1-2b_1, 1-2b_2}{2c}{\dfrac{z}{z+1}, \dfrac{z}{z-1}}$$
and
$$(1+z)^{1-2b_2-2c}(1-z)^{2b_2} \hF{2}{1}{b_2-b_1+c, b_1+b_2+c-\frac{1}{2}}{2c}{\dfrac{4z}{(z+1)^2}}.$$
\end{enumerate}
Relations between the corresponding functions over $\k$ are the following.
As a remark, (i) of the following theorem is essentially proved by He \cite[Corollary 2.2]{He2}.

\begin{thm}\label{F3 and 3F2, 2F1}\,
\begin{enumerate}
\item Suppose that $\e \not\in \{\a_1, \a_2, \b_2, \ol{\b_1\b_2}\c\}$.
Then, for any $\l \in \k$,
$$ \FBB{\a_1, \a_2\, ;\, \b_1,\b_2}{\c}{\l,1} = \dfrac{g^\circ(\c) g(\ol{\a_2\b_2}\c)}{g^\circ(\ol{\a_2}\c) g^\circ(\ol{\b_2}\c)} \hF{3}{2}{\a_1, \b_1, \ol{\a_2\b_2}\c}{\ol{\a_2}\c, \ol{\b_2}\c}{\l}.$$

\item Suppose that $\e \not\in\{ \a_1,\a_2, \b_1,\b_2, \ol{\a_1\a_2}\c \}$. 
Then, for $\l\neq1$, we have
\begin{align*}
& \FBB{\a_1,\a_2 \,;\, \b_1, \b_2}{\c}{\l, \dfrac{\l}{\l-1}}\\
& = \dfrac{g(\a_1\a_2\ol{\c}) g(\a_2\b_1\ol{\c}) g(\ol{\b_2})}{g(\ol{\c}) g^\circ (\a_1\a_2\b_1\ol{\c}) g^\circ(\a_2\ol{\b_2})} \ol{\c}(\l) \a_2(1-\l) \hF{3}{2}{\a_1\a_2\ol{\c}, \a_2\b_1\ol{\c}, \ol{\b_2}}{\a_1\a_2\b_1\ol{\c}, \a_2\ol{\b_2}}{1-\l}.
\end{align*}

\item Suppose that $p \neq 2$, $\e\not\in\{\b_1^2,\b_2^2\}$ and $(\b_1\b_2)^2, (\b_1\ol{\b_2})^2 \not\in \{\c^2, \ol{\c}^2\}$.
Then, for $\l \neq \pm 1$, we have
\begin{align*}
& \FBB{\b_1^2, \b_2^2\, ;\, \ol{\b_1}^2, \ol{\b_2}^2}{\c^2}{\dfrac{\l}{\l+1}, \dfrac{\l}{\l-1}}\\
& = \ol{\b_2^2\c^2}(1+\l) \b_2^2(1-\l) \sum_{\c' \in \{\c, \c\phi\}} \hF{2}{1}{\ol{\b_1}\b_2\c', \b_1\b_2\c'\phi}{\c^2}{\dfrac{4\l}{(\l+1)^2}}.
\end{align*}
\end{enumerate}
\end{thm}

Note the symmetry
\begin{align*}
\FBB{\a_1, \a_2\, ;\, \b_1, \b_2}{\c}{x,y}=\FBB{\b_1, \b_2\, ;\, \a_1, \a_2}{\c}{x,y}\quad (x,y \in \k).
\end{align*}
We use the following formula due to Otsubo \cite[Corollary 3.16]{Otsubo}: if $\a, \b \not\in\{\e, \c\}$ and $\l\neq0$, then
\begin{equation}
\hF{2}{1}{\a,\b}{\c}{\l} = \dfrac{g^\circ(\c) g(\a\b\ol{\c})}{g(\a)g(\b)}\a\ol{\c}(\l)\ol{\a\b}\c(1-\l)\hF{2}{1}{\ol\a, \ol\a\c}{\ol{\a\b}\c}{\dfrac{\l-1}{\l}}. \label{2F1 formula}
\end{equation}

\begin{proof}[Proof of Theorem \ref{F3 and 3F2, 2F1}]
When $\l =0$, the theorem is clear.
Thus, we prove for $\l \neq 0$.
By the author \cite[Theorem 3.5]{N}, if $\e\not\in\{\a_1,\a_2, \ol{\b_1\b_2}\c\}$, then for $x,y \neq 0$,
\begin{align}
&\dfrac{g(\b_1)g(\b_2)g(\ol{\b_1\b_2}\c) }{g^\circ(\c)}\FBB{\a_1,\a_2\, ;\, \b_1,\b_2}{\c}{x,y} \label{FB int.}\\
&\quad=\sum_{u, v\in\k^\times} \ol{\a_1}(1-xu)\ol{\a_2}(1-yv) \b_1(u) \b_2(v) \ol{\b_1\b_2}\c(1-u-v).\nonumber
\end{align}

(i) By \eqref{FB int.}, we have
\begin{align*}
& \FBB{\a_1,\a_2\, ;\, \b_1, \b_2}{\c}{\l,1} \\
& = \dfrac{g^\circ(\c)}{g(\b_1)g(\b_2)g(\ol{\b_1\b_2}\c)} \sum_{u,v} \ol{\a_1}(1-\l u)\ol{\a_2}(1-v) \b_1(u) \b_2(v) \ol{\b_1\b_2}\c(1-u-v).
\end{align*}
Putting $t = 1-v$ and $s = u/t$, and using \eqref{3F2 int.} (under the assumption $\e\not\in\{\a_1,\b_2,\ol{\b_1\b_2}\c\}$), the right-hand side above is equal to
\begin{align*}
& \dfrac{g^\circ(\c)}{g(\b_1)g(\b_2)g(\ol{\b_1\b_2}\c)} \sum_{s,t} \ol{\a_1}(1-\l st) \b_1(s) \ol{\b_1\b_2}\c(1-s) \ol{\a_2\b_2}\c(t) \b_2(1-t)\\
& =\dfrac{g^\circ(\c)}{g(\b_1)g(\b_2)g(\ol{\b_1\b_2}\c)}j(\b_1,\ol{\b_1\b_2}\c)j(\b_2,\ol{\a_2\b_2}\c) \hF{3}{2}{\a_1,\b_1,\ol{\a_2\b_2}\c}{\ol{\a_2}\c, \ol{\b_2}\c}{\l}.
\end{align*}
Thus, we obtain (i) by \eqref{J=G}.

(ii) By \eqref{FB int.}, we have (under the assumption $\e \not\in \{\b_1,\b_2,\ol{\a_1\a_2}\c\}$)
\begin{align*}
& \dfrac{g(\a_1)g(\a_2) g(\ol{\a_1\a_2}\c)}{g^\circ(\c)}\FBB{\a_1, \a_2\, ;\, \b_1, \b_2}{\c}{\l, \dfrac{\l}{\l-1}}\\
& = \sum_{u,v} \ol{\b_1}(1-\l u) \ol{\b_2}\Big( 1 - \frac{\l}{\l-1}v \Big) \a_1(u) \a_2(v) \ol{\a_1\a_2}\c(1-u-v).
\end{align*}
Putting $s = \frac{1}{1-\l u}$ and $t = 1-\frac{\l v}{\l-1}$, the right-hand side above is equal to
\begin{align*}
\a_1\a_2(-1)\ol{\c}(\l)\a_2(1-\l)\sum_{s,t} \ol{\a_1\a_2}\c(1-(1-\l)st) \a_2\b_1\ol{\c}(s) \a_1(1-s) \ol{\b_2}(t) \a_2(1-t).
\end{align*}
Thus, noting the assumption $\e\not\in\{\a_1,\a_2,\ol{\a_1\a_2}\c\}$, we obtain (ii) by \eqref{3F2 int.}, \eqref{J=G} and \eqref{Gauss sum thm}.

(iii) Using the relation in Remark \ref{FA=FB}, \eqref{Gauss sum thm} and \eqref{dup. form. gauss}, we have
\begin{align}
& \FBB{\b_1^2, \b_2^2\, ;\, \ol{\b_1}^2, \ol{\b_2}^2}{\c^2}{\dfrac{\l}{\l+1}, \dfrac{\l}{\l-1}} \label{eq.}\\
& = G \cdot \b_1^2\b_2^2(4\l) \ol{\b_1}^2(\l+1)\ol{\b_2}^2(\l-1) \FAA{\ol{\b_1^2\b_2^2\c^2}\, ;\, \ol{\b_1}^2, \ol{\b_2}^2}{\ol{\b_1}^4, \ol{\b_2}^4}{\dfrac{\l+1}{\l}, \dfrac{\l-1}{\l}},\nonumber
\end{align}
where
$$ G = (\ol\c^2)_{\ol{\b_1^2\b_2^2}}(\b_1^2)_{\b_1^2}(\b_2^2)_{\b_2^2} = \dfrac{g(\ol{\b_1^2\b_2^2\c^2})}{g(\ol\c^2)}\cdot \dfrac{g(\b_1^4)g(\b_2^4)}{g(\b_1^2)g(\b_2^2)}.$$
Applying to \eqref{dup. form. gauss} to $g(\b_1^4)$ and $g(\b_2^4)$ and using \eqref{Gauss sum thm},
$$G = \dfrac{g(\ol{\b_1^2\b_2^2\c^2}) g(\b_1^2\phi)g(\b_2^2\phi)}{q \phi(-1) g(\ol{\c}^2)}.$$
By Theorem \ref{F2 = 2F1}, we have
\begin{align*}
& \FAA{\ol{\b_1^2\b_2^2\c^2}\, ;\, \ol{\b_1}^2, \ol{\b_2}^2}{\ol{\b_1}^4, \ol{\b_2}^4}{\dfrac{\l+1}{\l}, \dfrac{\l-1}{\l}}\\
& = \b_1^2\b_2^2\c^2\Big( \dfrac{\l-1}{2\l}\Big) \sum_{\chi \in \{\ol{\b_1\b_2\c},\, \ol{\b_1\b_2\c}\phi\}} \dfrac{(\b_2^2\phi)_\chi}{(\phi)_\chi} \hF{2}{1}{\chi, \chi\phi\b_2^2}{\ol{\b_1}^2\phi}{\Big(\dfrac{\l+1}{\l-1}\Big)^2}.
\end{align*}
Using \eqref{2F1 formula} and \eqref{Gauss sum thm} and applying \eqref{dup. form. gauss} to $g(\phi)/g(\chi)g(\chi\phi)$ (noting $\chi^2 = \ol{\b_1^2\b_2^2\c^2}$ and the assumption $(\b_1\b_2)^2, (\b_1\ol{\b_2})^2\not\in\{\c^2,\ol\c^2\}$), we have 
\begin{align*}
&\b_1^2\b_2^2\c^2\Big( \dfrac{\l-1}{2\l}\Big) \sum_{\chi \in \{\ol{\b_1\b_2\c},\, \ol{\b_1\b_2\c}\phi\}} \dfrac{(\b_2^2\phi)_\chi}{(\phi)_\chi}  \hF{2}{1}{\chi, \chi\b_2^2\phi}{\ol{\b_1}^2\phi}{\Big(\dfrac{\l+1}{\l-1}\Big)^2}\\
& =\c^2(2) \ol{\b_1^2\b_2^2}(2\l) \b_2^4(\l-1) \b_1^2\ol{\b_2^2\c^2}(\l+1) \sum_\chi \dfrac{g(\phi) g^\circ(\ol{\b_1}^2\phi) g(\ol{\c}^2)}{g(\chi)g(\chi\phi) g(\b_2^2 \phi)}\hF{2}{1}{\ol{\chi}, \ol{\chi\b_1^2}\phi}{\c^2}{\dfrac{4\l}{(\l+1)^2}}\\
& = \ol{\b_1^2\b_2^2}(4\l) \b_2^4(\l-1) \b_1^2\ol{\b_2^2\c^2}(\l+1) G^{-1} \sum_\chi \hF{2}{1}{\ol{\chi}, \ol{\chi\b_1^2}\phi}{\c^2}{\dfrac{4\l}{(\l+1)^2}}.
\end{align*}
Thus, putting $\c' = \ol{\b_1\b_2\chi}\phi$, we have
\begin{align*}
& \FAA{\ol{\b_1^2\b_2^2\c^2}\, ;\, \ol{\b_1}^2, \ol{\b_2}^2}{\ol{\b_1}^4, \ol{\b_2}^4}{\dfrac{\l+1}{\l}, \dfrac{\l-1}{\l}}\\
& = G^{-1} \ol{\b_1^2\b_2^2}(4\l) \b_2^4(\l-1) \b_1^2\ol{\b_2^2\c^2}(\l+1) \sum_{\c'\in\{\c,\c\phi\}} \hF{2}{1}{\ol{\b_1}\b_2\c', \b_1\b_2\c'\phi}{\c^2}{\dfrac{4\l}{(\l+1)^2}}.
\end{align*}
Hence, we obtain (iii) by \eqref{eq.}.
\end{proof}

Recall the formula due to Karlsson (cf. \cite[p. 302 (91)]{Srivastava-Karlsson})
\begin{align*}
& (1-z)^{a_1+a_2+b_1-c} \FBB{a_1, a_2 \, ;\,  b_1, b_2}{c}{z, \dfrac{z}{z-1}}\\
& = \FBB{c-a_2-b_1, a_1+a_2+b_1+b_2-c \, ;\, c-a_1-a_2, a_2}{c}{z, \dfrac{z}{z-1}}.
\end{align*}
An analogue of this formula is the following
\begin{cor}
Suppose that $\c \not\in\{ \a_1\a_2, \b_1\b_2, \a_1\b_2, \a_2\b_1, \a_1\a_2\b_1\b_2\}$ and $\e \not\in\{\a_i, \b_i\, \mid i=1,2\}$.
Then, for $\l \neq 1$, we have
\begin{align*}
& \a_1\a_2\b_1\ol{\c}(1-\l) \FBB{\a_1, \a_2\, ;\, \b_1, \b_2}{\c}{\l, \dfrac{\l}{\l-1}}\\
& = \FBB{\ol{\a_2\b_1}\c, \a_1\a_2\b_1\b_2\ol{\c} \, ;\, \ol{\a_1\a_2}\c, \a_2}{\c}{\l, \dfrac{\l}{\l-1}}.
\end{align*}
\end{cor}

\begin{proof}
By Theorem \ref{F3 and 3F2, 2F1} (ii), the two functions
\begin{align*}
\c(\l)\ol{\b_2}(1-\l) \FBB{\b_1, \b_2\, ;\, \a_1, \a_2}{\c}{\l, \dfrac{\l}{\l-1}}
\end{align*}
and
\begin{align*}
\c(\l) \ol{\a_1\a_2\b_1\b_2}\c(1-\l)\FBB{\ol{\a_2\b_1}\c, \a_1\a_2\b_1\b_2\ol{\c} \, ;\, \ol{\a_1\a_2}\c, \a_2}{\c}{\l, \dfrac{\l}{\l-1}}
\end{align*}
are both equal to
$$\dfrac{g(\b_1\b_2\ol{\c}) g(\a_1\b_2\ol{\c})g(\ol{\a_2})}{g(\ol{\c}) g^\circ(\a_1\b_1\b_2\ol{\c})g^\circ(\ol{\a_2}\b_2)}\hF{3}{2}{\b_1\b_2\ol{\c}, \a_1\b_2\ol{\c}, \ol{\a_2}}{\a_1\b_1\b_2\ol{\c}, \ol{\a_2}\b_2}{1-\l}.$$
Thus, we obtain the corollary.
\end{proof}

\subsection{Reduction and transformation formulas for Appell's functions $F_4$ over $\k$}
Over $\C$, Vid\=unas \cite[Theorem 4.3]{Vidunas} proves that the following four functions satisfy the same ordinary differential equation of order 3:
\begin{align*}
& \FCC{a\, ; \, b}{c+\frac{1}{2}, \frac{1}{2}}{z^2, (1-z)^2}, \quad (1-z)\FCC{a+\frac{1}{2} \, ;\, b+\frac{1}{2}}{c+\frac{1}{2}, \frac{3}{2}}{z^2, (1-z)^2},\\
& (1-z)^{-2a} \FCC{a\, ;\, a+\frac{1}{2}}{c+\frac{1}{2}, 1+a-b}{\dfrac{z^2}{(1-z)^2}, \dfrac{1}{(1-z)^2}} \quad \mathrm{and}\quad \hF{3}{2}{2a, 2b, c}{a+b+\frac{1}{2}, 2c}{z}.
\end{align*}
Over $\k$, we have the following formulas.
\begin{thm}\label{F4 reduction}Suppose that $p \neq 2$.
\begin{enumerate}
\item We have
\begin{align*}
\FCC{\a\, ; \, \b}{\c\phi, \phi}{\l^2, (1-\l)^2} = (\a)_\phi (\b)_\phi \FCC{\a\phi\, ;\, \b\phi}{\c\phi, \phi}{\l^2, (1-\l)^2}.
\end{align*}

\item Suppose that $\e \not\in\{\a, \ol{\b}\c\}$.
Then, for $\l \neq 1$, we have
\begin{align*}
& \FCC{\a\, ; \, \b}{\c\phi, \phi}{\l^2, (1-\l)^2} \\
& = \dfrac{g( \ol{\a}\b) g(\a\phi)}{g(\b) g(\phi)} \ol{\a}^2(1-\l) \FCC{\a\, ;\, \a\phi}{\c\phi, \a\ol{\b}}{\dfrac{\l^2}{(1-\l)^2}, \dfrac{1}{(1-\l)^2}}.
\end{align*}

\item Suppose that $\e \not\in\{\a\ol{\b}\phi, \c\}$.
Then, for $\l \neq 1$, we have
\begin{align*}
& \ol{\a}^2(1-\l) \FCC{\a\, ;\, \a\phi}{\c\phi, \a\ol{\b}}{\dfrac{\l^2}{(1-\l)^2}, \dfrac{1}{(1-\l)^2}}\\
& = \dfrac{g(\b)g(\b\phi)}{g(\ol{\a}\b)g^\circ(\a\b\phi)} \hF{3}{2}{\a^2, \b^2, \c}{\a\b\phi, \c^2}{\l}.
\end{align*}
\end{enumerate}
\end{thm}

We use the following proposition.
\begin{prop}\label{F4 line. trans.}
Suppose that $\e \not \in \{\a, \ol{\b}\c_1\c_2\}$. 
Then, for any $x \in \k$ and $y \in \k^\times$, we have
\begin{align*}
& \FCC{\a\, ;\, \b}{\c_1, \c_2}{x,y} = \dfrac{g(\ol{\a}\b) g(\a\ol{\c_2})}{g(\b)g(\ol{\c_2})} \ol{\a}(y) \FCC{\a\, ;\, \a\ol{\c_2}}{\c_1, \a\ol{\b}}{\dfrac{x}{y}, \dfrac{1}{y}}.
\end{align*}
\end{prop}

\begin{proof}
When $x=0$ it is clear.
When $x\neq 0$, the author \cite[Theorem 3.7]{N} proves that, if $\e \not \in \{\a, \ol{\b}\c_1\c_2\}$, then
\begin{align*}
& \dfrac{g(\ol{\c_1}) g(\ol{\c_2}) g(\ol{\b}\c_1\c_2)}{g^\circ(\ol{\b})}\FCC{\a\, ;\, \b}{\c_1, \c_2}{x,y}\\
& = \sum_{u,v\in\k^\times} \ol{\a} \Big(1-\frac{x}{u}-\frac{y}{v}\Big) \ol{\c_1}(u)\ol{\c_2}(v) \ol{\b}\c_1\c_2( 1-u-v ).
\end{align*}
Putting $s = -u/v$ and $t = 1/v$ and using the formula above again, the right-hand side above is equal to
\begin{align*}
& \a\b\c_2(-1) \ol{\a}(y) \sum_{s,t} \ol{\a} \Big( 1 - \frac{x}{ys} - \frac{1}{yt}\Big) \ol{\c_1}(s) \ol{\a}\b(t) \ol{\b}\c_1\c_2(1-s-t)\\
& = \a\b\c_2(-1) \ol{\a}(y) \dfrac{g(\ol{\c_1}) g(\ol{\a}\b) g(\ol{\b}\c_1\c_2)}{g^\circ(\ol{\a}\c_2)} \FCC{\a\, ;\, \a\ol{\c_2}}{\c_1, \a\ol{\b}}{\dfrac{x}{y}, \dfrac{1}{y}}. 
\end{align*}
Thus, we obtain the proposition by \eqref{Gauss sum thm}.
\end{proof}

\begin{proof}[Proof of Theorem \ref{F4 reduction}]
By putting $\eta =\mu\phi$ and \eqref{Poch formula}, 
\begin{align*}
\FCC{\a\, ; \, \b}{\c\phi, \phi}{\l^2, (1-\l)^2} 
& = \dfrac{1}{(1-q)^2}\sum_{\nu, \mu}\dfrac{(\a)_{\nu\mu} (\b)_{\nu\mu}}{(\e)_\nu^\circ(\c\phi)_\nu^\circ (\e)_\mu^\circ (\phi)_\mu^\circ}\nu^2(\l)\mu^2(1-\l)\\
& = \dfrac{1}{(1-q)^2}\sum_{\nu, \eta}\dfrac{(\a)_{\nu\eta\phi} (\b)_{\nu\eta\phi}}{(\e)_\nu^\circ(\c\phi)_\nu^\circ (\e)_{\eta\phi}^\circ (\phi)_{\eta\phi}^\circ}\nu^2(\l)\eta^2(1-\l)\\
& = \dfrac{(\a)_\phi (\b)_\phi}{(\e)_\phi^\circ (\phi)_\phi^\circ} \FCC{\a\phi\, ;\, \b\phi}{\c\phi, \phi}{\l^2, (1-\l)^2}.
\end{align*}
Noting that $(\e)_\phi^\circ (\phi)_\phi^\circ=1$, we obtain (i).
(ii) can be easily shown by Proposition \ref{F4 line. trans.}.

By Theorem \ref{FA and FC integral rep.} (iii) and Proposition \ref{Kummer second}, we have
\begin{align*}
& \ol{\a}^2(1-\l)\FCC{\a\, ;\, \a\phi}{\c\phi, \a\ol{\b}}{\dfrac{\l^2}{(1-\l)^2}, \dfrac{1}{(1-\l)^2}}\\
& = -\dfrac{\ol{\a}^2(1-\l)}{g(\a^2)} \sum_t \a^2(t) \hF{1}{1}{\c}{\c^2}{-\dfrac{2\l t}{1-\l}} \hF{1}{1}{\a\ol{\b}\phi}{\a^2\ol{\b}^2}{\dfrac{2t}{1-\l}}\\
& = -\dfrac{\ol{\a}^2(1-\l)}{(1-q)^2 g(\a^2)} \sum_{\nu,\mu} \dfrac{(\c)_\nu}{(\e)^\circ_\nu (\c^2)_\nu^\circ}\cdot \dfrac{(\a\ol{\b}\phi)_\mu}{ (\e)_\mu^\circ (\a^2\ol{\b}^2)_\mu^\circ}\nu(-\l)\ol{\nu\mu}\Big(\dfrac{1-\l}{2}\Big) \sum_t \a^2\nu\mu(t)\\
& = \dfrac{\ol{\a}(4)}{(1-q) g(\a^2)} \sum_\nu \dfrac{(\c)_\nu}{(\e)_\nu^\circ (\c^2)_\nu^\circ } \cdot \dfrac{(\a\ol{\b}\phi)_{\ol{\a^2\nu}}}{(\e)_{\ol{\a^2\nu}}^\circ (\a^2\ol{\b}^2)_{\ol{\a^2\nu}}^\circ}\nu(-\l).
\end{align*}
Using \eqref{inversion Poch} and \eqref{Poch formula}, the right-hand side above is equal to
\begin{align*}
& \dfrac{\ol{\a}(4)}{(1-q) g(\a^2)} \sum_\nu \dfrac{ (\c)_\nu }{(\e)_\nu^\circ (\c^2)_\nu^\circ } \cdot\dfrac{ (\e)_{\a^2\nu} (\ol{\a}^2\b^2)_{\a^2\nu} }{ (\ol{\a}\b\phi)_{\a^2\nu}^\circ}\nu(\l)\\
& = \dfrac{(\ol{\a}^2\b^2)_{\a^2}}{(\ol{\a}\b\phi)_{\a^2}^\circ}\ol{\a}(4) \hF{3}{2}{\a^2, \b^2, \c}{\a\b\phi, \c^2}{\l}.
\end{align*}
Thus, we obtain (iii) by applying the duplication formula \eqref{dup. form. Poch} to $(\ol\a^2\b^2)_{\a^2}$.
\end{proof}

Recall the formula due to Appell-Kamp\'e de F\'eriet (\cite[p. 303 (97)]{Srivastava-Karlsson}):
\begin{align*}
\FCC{a\, ;\, a-b+\frac{1}{2}}{c, b+\frac{1}{2}}{x, y^2} = (1+y)^{-2a} \FAA{a\, ;\, a-b+\frac{1}{2}, b}{c, 2b}{\dfrac{x}{(1+y)^2}, \dfrac{4y}{(1+y)^2}}.
\end{align*}
The following is a finite field analogue of this formula.
\begin{thm}\label{analogue of Appell-Kampe}
Suppose that $p \neq 2$ and $\b^2\neq\e$. 
Then, for $x\in\k$ and $y \in \k-\{\pm 1\}$, we have
\begin{align*}
& \FCC{\a\, ;\, \a\ol{\b}\phi}{\c, \b\phi}{x, y^2} = \ol{\a}^2(1+y) \FAA{\a \, ;\, \a\ol{\b}\phi, \b}{\c, \b^2}{\dfrac{x}{(1+y)^2}, \dfrac{4y}{(1+y)^2}}.
\end{align*}

\end{thm}

We use the following lemma, which is a slightly generalization of Otsubo's formula \cite[Theorem 5.4]{Otsubo}.
\begin{lem}\label{Gauss quad.} If $\l \neq \pm 1$ and $\b^2 \neq \e$, then
\begin{equation*}
\hF{2}{1}{\a, \a\ol{\b}\phi}{\b\phi}{\l^2} = \ol{\a}^2(1+\l) \hF{2}{1}{\a, \b}{\b^2}{\dfrac{4\l}{(1+\l)^2}},
\end{equation*}
\end{lem}
\begin{proof}
When $\a \not\in\{\e, \b\phi, \b^2\}$ and $\b\neq\e$, Otsubo \cite[Theorem 5.4]{Otsubo} proves the formula. 
We see that the formula is true even if $\a \in \{\e, \b\phi, \b^2\}$, when $\l \neq \pm 1$ and $\b^2 \neq \e$ (i.e. $\e\not\in\{\b,\b\phi\}$).
Indeed, when $\a =\e$,
\begin{align*}
\hF{2}{1}{\a, \a\ol{\b}\phi}{\b\phi}{\l^2} 
& = \dfrac{q}{1-q} \sum_\nu q^{-\d(\nu)}\dfrac{(\ol\b\phi)_\nu}{(\b\phi)_\nu^\circ}\nu(\l^2)\\
& = \dfrac{q}{1-q} \sum_\nu \dfrac{(\ol\b\phi)_\nu}{(\b\phi)_\nu^\circ}\nu(\l^2) + 1.
\end{align*}
Taking $\nu=\ol\b\phi\mu$ and using \eqref{Poch formula}, \eqref{1F0 int.}, \eqref{dup. form. gauss} and \eqref{J=G}, we have
\begin{align*}
\dfrac{q}{1-q} \sum_\nu \dfrac{(\ol\b\phi)_\nu}{(\b\phi)_\nu^\circ}\nu(\l^2)
& = \dfrac{q}{1-q} \sum_\mu \dfrac{(\ol\b\phi)_{\ol\b\phi\mu}}{(\b\phi)_{\ol\b\phi\mu}^\circ}\ol\b\phi\mu(\l^2)\\
& = q\dfrac{(\ol\b\phi)_{\ol\b\phi}}{(\b\phi)_{\ol\b\phi}^\circ}\ol\b^2(\l)\hF{1}{0}{\ol\b^2}{}{\l^2}\\
& = \dfrac{g(\ol\b^2) g(\b\phi)}{g(\ol\b\phi)} \ol{\b}^2(\l)\ol\b^2(1-\l^2)\\
& = j(\ol{\b}, \b\phi) \ol{\b}^2(2\l) \b^2(1-\l^2).
\end{align*}
On the other hand, when $\a=\e$, we similarly have
\begin{align*}
\ol{\a}^2(1+\l) \hF{2}{1}{\a, \b}{\b^2}{\dfrac{4\l}{(1+\l)^2}}
& = q\dfrac{(\b)_{\ol\b^2}}{(\b^2)_{\ol\b^2}^\circ}\ol\b^2\Big(\dfrac{4\l}{(1+\l)^2}\Big) \hF{1}{0}{\ol\b}{}{\dfrac{4\l}{(1+\l)^2}}+1\\
& = \dfrac{g(\ol\b)g(\b^2)}{g(\b)}\ol\b^2(4\l)\b^2(1-\l^2)+1\\
& = j(\ol{\b}, \b\phi) \ol{\b}^2(2\l) \b^2(1-\l^2) + 1.
\end{align*}
Thus, the both sides of the lemma are equal to $j(\ol{\b}, \b\phi) \ol{\b}^2(2\l) \b^2(1-\l^2) + 1$.

When $\a = \b^2$, using \eqref{dup. form. gauss} and \eqref{J=G}, we have 
\begin{align*}
\hF{2}{1}{\a, \a\ol{\b}\phi}{\b\phi}{\l^2} 
& = \dfrac{1}{1-q}\sum_\nu q^{-(\b\phi\nu)}\dfrac{(\b^2)_\nu}{(\e)_\nu^\circ}\nu(\l^2)\\
& = \hF{1}{0}{\b^2}{}{\l^2} + \dfrac{1}{q}\cdot \dfrac{(\b^2)_{\ol\b\phi}}{(\e)^\circ_{\ol\b\phi}}\ol\b^2(\l)\\
& = \ol\b^2(1-\l^2) + \dfrac{g(\b\phi)}{g(\b^2)g^\circ(\ol\b\phi)}\ol\b^2(\l)\\
& =  \ol\b^2(1-\l^2) + j(\b,\ol\b\phi)^{-1}\ol\b^2(2\l).
\end{align*}
Similarly, we can also check that the right-hand side of the lemma is equal to $\ol\b^2(1-\l^2) + j(\b,\ol\b\phi)^{-1}\ol\b^2(2\l)$ when $\a = \b^2$.

When $\a = \b\phi$, we have (note that $\a \neq \e$ and $\l \neq \pm 1$ by assumptions)
\begin{align*}
\hF{2}{1}{\a, \a\ol{\b}\phi}{\b\phi}{\l^2}
& = \dfrac{1}{1-q} \sum_\nu q^{1-\d(\nu)-\d(\a\nu)} \nu(\l^2) \\
& = \dfrac{1}{1-q} \sum_\nu \nu(\l^2) + 1+\ol\a^2(\l)
= 1 + \ol\a^2(\l).
\end{align*}
On the other hand, using \eqref{dup. form. gauss}, \eqref{Gauss sum thm} and \eqref{J=G}, we have (when $\a=\b\phi$)
\small
\begin{align*}
\hF{2}{1}{\a, \b}{\b^2}{\dfrac{4\l}{(1+\l)^2}} 
& = \dfrac{1}{1-q} \sum_\nu \dfrac{(\a^2)_{\nu^2}}{(\e)_\nu^\circ (\a^2)_\nu^\circ} \nu\Big(\dfrac{\l}{(1+\l)^2} \Big)\\
& = \dfrac{1}{1-q} \sum_\nu j(\ol{\a^2\nu},\ol\nu) \nu\Big(\dfrac{\l}{(1+\l)^2} \Big)\\
& = \dfrac{-1}{1-q} \sum_t \ol\a^2(t) \sum_\nu \ol\nu\Big( \dfrac{t(1-t)(1+\l)^2}{\l}\Big)
 = \a^2(1+\l)\big(1+\ol\a^2(\l)\big).
\end{align*}
\normalsize
Here, note that $\dfrac{t(1-t)(1+\l)^2}{\l}=1$ if and only if $t = 1/(1+\l)$ or $t = \l/(1+\l)$.
Thus, the right-hand side of the lemma is equal to $1 + \ol\a^2(\l)$ when $\a = \b\phi$, and hence, we obtain the lemma. 
\end{proof}

\begin{proof}[Proof of Theorem \ref{analogue of Appell-Kampe}]
By \eqref{Poch formula} and \eqref{Gauss quad.}, we have
\begin{align*}
 & \FCC{\a\, ; \, \a\ol{\b}\phi}{\c, \b\phi}{x, y^2} \\
 & = \frac{1}{1-q} \sum_{\nu} \frac{(\a)_\nu (\a\ol{\b}\phi)_\nu}{(\e)_\nu^\circ (\c)_\nu^\circ}\nu(x) \hF{2}{1}{\a\nu, \a\nu\ol{\b}\phi}{\b\phi}{y^2}\\
 & = \ol{\a}^2(1+y) \frac{1}{1-q} \sum_{\nu} \frac{(\a)_\nu (\a\ol{\b}\phi)_\nu}{(\e)_\nu^\circ (\c)_\nu^\circ}\nu\Big( \frac{x}{(1+y)^2}\Big) \hF{2}{1}{\a\nu, \b}{\b^2}{\dfrac{4y}{(1+y)^2}}\\
 & = \ol{\a}^2(1+y) \frac{1}{(1-q)^2} \sum_{\nu,\mu} \frac{(\a)_\nu(\a\nu)_\mu (\a\ol{\b}\phi)_\nu (\b)_\mu}{(\e)_\nu^\circ (\c)_\nu^\circ (\e)_\mu^\circ(\b^2)_\mu^\circ} \nu\Big( \frac{x}{(1+y)^2}\Big) \mu\Big( \dfrac{4y}{(1+y)^2} \Big).
 \end{align*}
Thus, we obtain the theorem by applying \eqref{Poch formula} to $(\a)_\nu(\a\nu)_\mu$.
\end{proof}

\subsection{Formulas between Appell's functions and ${}_4F_3$-functions over $\k$}
Over $\C$, the following identities are known due to Bailey, Burchnall and Srivastava (cf. \cite[(4.4), (4.5) and (4.3)]{Bailey} and \cite[(8) and (9)]{I Srivastava}):
\begin{align*}
& \FAA{2a\, ;\, 2b, 2c}{4b, 4c}{z, -z} = \hF{4}{3}{a, a+\frac{1}{2}, b+c, b+c+\frac{1}{2}}{2b+\frac{1}{2}, 2c+\frac{1}{2}, 2b+2c}{z^2},\\
& \FAA{2a\, ;\, b, b}{2c, 2c}{z, -z} = \hF{4}{3}{a, a+\frac{1}{2}, b, 2c-b}{2c, c, c+\frac{1}{2}}{z^2},\\
& \FBB{2a, 2a\, ;\, 2b, 2b}{2c}{z, -z} = \hF{4}{3}{a+b, a+b+\frac{1}{2}, 2a, 2b}{2a+2b, c, c+\frac{1}{2}}{z^2},\\
& \FCC{a\, ;\, b}{2c_1, 2c_2}{z,z} = \hF{4}{3}{a, b, c_1+c_2, c_1+c_2-\frac{1}{2}}{2c_1,2c_2,2c_1+2c_2-1}{4z},
\end{align*}
and
$$\FCC{2a\, ;\, 2b}{2c,2c}{z,-z} = \hF{4}{3}{a, a+\frac{1}{2}, b, b+\frac{1}{2}}{2c, c, c+\frac{1}{2}}{-4z^2}.$$
Tripathi-Barman \cite{TB2} give finite field analogues of them (their functions $F_1^*$--$F_4^*$ over $\k$ coincide with our $F_1$--$F_4$, respectively, by similar observation as in \cite[Remark 2.13 (iv)]{Otsubo}).
In this subsection, we give short proofs of the analogues under weaker conditions by using Theorem \ref{FA and FC integral rep.} and product formulas for ${}_1F_1$, ${}_2F_0$ and ${}_0F_1$ due to Otsubo \cite{Otsubo} and Otsubo-Senoue \cite{Otsubo2}.
Furthermore, we obtain some additional formulas between Appell's functions and ${}_4F_3$-functions over $\k$.

\begin{thm}[{\cite[Corollary 1.3, Theorems 1.4, 1.5, 1.7 and 1.6]{TB2}}]\label{TB formula}
Suppose that $p \neq 2$.
\begin{enumerate}
\item Suppose that $\e \not\in \{ \b^2, \c^2, \b^2\c^2, \b^2\ol{\c}^2 \}$.
We have
$$\FAA{\a^2\, ;\, \b^2, \c^2}{\b^4, \c^4}{\l, -\l} = \hF{4}{3}{\a,\a\phi, \b\c, \b\c\phi}{\b^2\phi, \c^2\phi, \b^2\c^2}{\l^2}.$$

\item Suppose that $\b \not\in\{ \e, \c, \c\phi, \c^2 \}$.
We have
$$\FAA{\a^2\, ;\, \b,\b}{\c^2, \c^2}{\l, -\l} = \hF{4}{3}{\a, \a\phi, \b, \ol{\b}\c^2}{\c^2, \c, \c\phi}{\l^2}.$$

\item Suppose that $\e \not\in \{\a^2, \b^2, \a^2\b^2\}$.
We have
$$\FBB{\a^2, \a^2\, ;\, \b^2, \b^2}{\c^2}{\l, -\l} = \hF{4}{3}{\a\b, \a\b\phi, \a^2, \b^2}{\a^2\b^2, \c, \c\phi}{\l^2}.$$

\item Suppose that $\e \not\in\{ \c_1^2\c_2^2, \c_1^2\ol{\c_2^2} \}$.
We have
$$\FCC{\a\, ;\, \b}{\c_1^2, \c_2^2}{\l, \l} = \hF{4}{3}{\a,\b,\c_1\c_2, \c_1\c_2\phi}{\c_1^2, \c_2^2, \c_1^2\c_2^2}{4\l}.$$

\item For any $\a,\b,\c \in \khat$, we have
$$\FCC{\a^2\, ;\, \b^2}{\c^2, \c^2}{\l,-\l} = \hF{4}{3}{\a,\a\phi,\b,\b\phi}{\c^2, \c, \c\phi}{-4\l^2}.$$
\end{enumerate}
\end{thm}

\begin{proof}
By Theorem \ref{FA and FC integral rep.} (i), we have
$$\FAA{\a^2\, ;\, \b^2,\c^2}{\b^4, \c^4}{\l, -\l} = -\dfrac{1}{g(\a^2)} \sum_t \psi(t) \a^2(t) \hF{1}{1}{\b^2}{\b^4}{\l t} \hF{1}{1}{\c^2}{\c^4}{-\l t}.$$
Using \cite[Corollary 3.3 (v)]{Otsubo2}, the right-hand side above is equal to
\begin{align*}
& -\dfrac{1}{g(\a^2)} \sum_t \psi(t) \a^2(t) \hF{2}{3}{\b\c, \b\c\phi}{\b^2\phi, \c^2\phi, \b^2\c^2}{\dfrac{\l^2 t^2}{4}}\\
& = -\dfrac{1}{1-q}\cdot \dfrac{1}{g(\a^2)} \sum_\nu \dfrac{(\b\c)_\nu (\b\c\phi)_\nu}{(\e)_\nu^\circ (\b^2\phi)_\nu^\circ (\c^2\phi)_\nu^\circ (\b^2\c^2)_\nu^\circ} \nu\Big(\dfrac{\l^2}{4}\Big) \sum_t \psi(t) \a^2\nu^2(t)\\
& = \dfrac{1}{1-q} \sum_\nu \dfrac{(\a^2)_{\nu^2} (\b\c)_\nu (\b\c\phi)_\nu}{(\e)_\nu^\circ (\b^2\phi)_\nu^\circ (\c^2\phi)_\nu^\circ (\b^2\c^2)_\nu^\circ} \nu\Big(\dfrac{\l^2}{4}\Big).
\end{align*}
Thus, we obtain (i) by applying the duplication formula \eqref{dup. form. Poch} to $(\a^2)_{\nu^2}$.
Similarly, we can prove (ii) by Theorem \ref{FA and FC integral rep.} (i), \cite[Theorem 6.2]{Otsubo} and \eqref{dup. form. Poch}.
We can also prove (iii) by Theorem \ref{FA and FC integral rep.} (ii), \cite[Corollary 3.3 (ii)]{Otsubo2}, \eqref{Gauss sum thm} and \eqref{dup. form. Poch}.

By Theorem \ref{FA and FC integral rep.} (iv), we have
$$\FCC{\a\, ;\, \b}{\c_1^2, \c_2^2}{\l, \l} = \dfrac{1}{g(\a) g(\b)} \sum_{s, t} \psi(s+t)\a(s)\b(t) \hF{0}{1}{}{\c_1^2}{\l st} \hF{0}{1}{}{\c_2^2}{\l st}.$$
Using \cite[Theorem 3.1]{Otsubo2}, the right-hand side above is equal to
\begin{align*}
& \dfrac{1}{g(\a) g(\b)} \sum_{s, t} \psi(s+t)\a(s)\b(t) \hF{2}{3}{\c_1\c_2, \c_1\c_2\phi}{\c_1^2, \c_2^2, \c_1^2\c_2^2}{4\l st}\\
& = \dfrac{1}{1-q}\cdot\dfrac{1}{g(\a) g(\b)}\sum_\nu \dfrac{(\c_1\c_2)_\nu (\c_1\c_2\phi)_\nu}{(\e)_\nu^\circ (\c_1^2)_\nu^\circ (\c_2^2)_\nu^\circ (\c_1^2\c_2^2)_\nu^\circ}\nu(4\l) \sum_{s}\psi(s)\a\nu(s) \sum_t \psi(t)\b\nu(t)\\
& = \dfrac{1}{1-q}\sum_\nu \dfrac{(\a)_\nu (\b)_\nu(\c_1\c_2)_\nu (\c_1\c_2\phi)_\nu}{(\e)_\nu^\circ (\c_1^2)_\nu^\circ (\c_2^2)_\nu^\circ (\c_1^2\c_2^2)_\nu^\circ}\nu(4\l).
\end{align*}
Thus, we obtain (iv).
Similarly, we can prove (v) by Theorem \ref{FA and FC integral rep.} (iv), \cite[Theorem 3.2]{Otsubo2} and \eqref{dup. form. Poch}.
\end{proof}

We obtain the following additional formulas.
(i) and (ii) are finite field analogues of Bailey's formula \cite[(4.6) and (4.2)]{Bailey}.
For (iii), as far as the author knows, the corresponding formula over $\C$ are not known.
\begin{thm} Suppose that $p \neq 2$.
\begin{enumerate}
\item Suppose that $\e \not\in\{\b,\b\ol{\c}^2,\b^2\ol{\c}^2 \}$.
We have
$$\FAA{\a^2\, ;\, \b, \b\ol{\c}^2}{\c^2,\ol{\c}^2}{\l, -\l} = \hF{4}{3}{\a,\a\phi, \b\ol{\c}\phi, \ol{\b}\c\phi}{\phi, \c\phi, \ol{\c}\phi}{\l^2}.$$

\item Suppose that $\e \not\in \{ \a^2, \b^2, \a^2\b^2, \a^2\ol{\b}^2 \}$.
We have
$$\FBB{\a^2, \b^2\, ;\, \ol{\a}^2, \ol{\b}^2}{\c^2}{\l, -\l} = \hF{4}{3}{\a\b, \ol{\a\b}, \a\ol{\b}\phi, \ol{\a}\b\phi}{\phi, \c, \c\phi}{\l^2}.$$

\item For any $\a, \b, \c\in\khat$, we have
$$\FCC{\a^2\, ;\, \b^2}{\c^2, \ol{\c}^2}{\l, -\l} = q^{\d(\c)} \hF{4}{3}{\a,\a\phi, \b, \b\phi}{\phi, \c\phi, \ol{\c}\phi}{-4\l^2}.$$
\end{enumerate}
\end{thm}

\begin{proof}
Similarly to Theorem \ref{TB formula} (i), we can prove (i) by Theorem \ref{FA and FC integral rep.} (i), \cite[Corollary 3.3 (iv)]{Otsubo2} and \eqref{dup. form. Poch}.
We can also prove (ii) by Theorem \ref{FA and FC integral rep.} (ii), \cite[Corollary 3.3 (iii)]{Otsubo2}, \eqref{Gauss sum thm} and \eqref{dup. form. Poch}.
Similarly to Theorem \ref{TB formula} (iv), we can prove (iii) by Theorem \ref{FA and FC integral rep.} (iv), \cite[Corollary 3.3 (i)]{Otsubo2} and \eqref{dup. form. Poch}.
\end{proof}

\subsection{Finite field analogues of reduction formulas for $F_D$}
                      
The following is a finite field analogue of the multinomial theorem for the Pochhammer symbols.
\begin{lem}\label{multinomial}
Suppose that $\b_1\cdots\b_n \neq \e$. 
Then, for any $\nu \in \khat$,
$$\dfrac{(\b_1\cdots\b_n)_\nu}{(\e)_\nu^\circ} = \dfrac{1}{(1-q)^{n-1}} \sum_{\nu_1\cdots\nu_n = \nu} \prod_{i=1}^n \dfrac{(\b_i)_{\nu_i}}{(\e)_{\nu_i}^\circ}.$$
\end{lem}

\begin{proof}
Put
$$f(\l) = \hF{1}{0}{\b_1\cdots\b_n}{}{\l} \quad \mathrm{and}\quad g(\l) = \prod_{i=1}^n \hF{1}{0}{\b_i}{}{\l}.$$
One shows that $-\widehat{f}(\nu)$ is the left-hand side of the lemma and $-\widehat{g}(\nu)$ is the right-hand side of the lemma for each $\nu \in \khat$.
By \eqref{1F0 int.}, for $ \l \neq 1$, we have
$g(\l) = \prod_{i=1}^n \b_i(1-\l) = \b_1\cdots\b_n(1-\l) = f(\l).$
For $\l=1$, noting that $\hF{1}{0}{\a}{}{1}=0$ (when $\a\neq \e$), we have if $\b_1\cdots\b_n\neq\e$ (hence $\b_i\neq \e$ for some $i$), then $f(1)=g(1)=0.$
Thus, $f(\l) = g(\l)$ for any $\l\in\k^\times$, and hence $\widehat{f}(\nu) = \widehat{g}(\nu)$ for all $\nu \in \khat$.
\end{proof}

Put $F_D^{(0)} = 1$.
Over $\C$, it is well known that (cf. \cite{Lauricella}), for a fixed $i = 0, 1, \dots, n-1$, 
\begin{align*}
&\FD{n}{a\, ;\,  b_1, \dots, b_n}{c}{z_1, \dots, z_i, x, \dots, x} \\
& = \FD{i+1}{a\, ;\, b_1, \dots, b_i, \sum_{j=i+1}^n b_j}{c}{z_1, \dots, z_i, x},
\end{align*}
and 
\begin{align*}
& \FD{n}{a\, ;\, b_1, \dots, b_n}{c}{z_1, \dots, z_i, 1, \dots, 1}\\
& = \dfrac{\Gamma(c) \Gamma(c-a-\sum_{j=i+1}^n b_j)}{\Gamma(c-a) \Gamma(c-\sum_{j=i+1}^n b_j)} \FD{i}{a\, ;\, b_1, \dots, b_i}{c-\sum_{j=i+1}^n b_j}{z_1, \dots, z_i}.
\end{align*}
The followings are finite field analogues of them.

\begin{thm}\label{FD red.}$ $
\begin{enumerate}
\item For $i = 0,\dots,n-1$, if $\b_{i+1}\cdots\b_n \neq \e$, then
\begin{align*}
& \FD{n}{\a\, ;\, \b_1,\dots,\b_n}{\c}{\l_1,\dots,\l_i,x,\dots,x}\\
& = \FD{i+1}{\a\, ;\, \b_1,\dots, \b_i, \b_{i+1}\cdots\b_n}{\c}{\l_1,\dots,\l_i,x}.
\end{align*}
In particular, if $ \b_{1}\cdots\b_n \neq \e$, 
\begin{align*}
& \FD{n}{\a\, ;\, \b_1,\dots,\b_n}{\c}{x,\dots,x} = \hF{2}{1}{\a, \b_{1}\cdots\b_n}{\c}{x}.
\end{align*}

\item For $i = 0,\dots,n-1$, if $\b_{i+1}\cdots\b_n \not\in \{ \e, \ol{\a}\c \}$, then
\begin{align*}
& \FD{n}{\a\, ;\, \b_1,\dots,\b_n}{\c}{\l_1,\dots,\l_i,1,\dots,1}\\
& = \dfrac{g^\circ(\c) g(\ol{\a\b_{i+1}\cdots\b_n}\c)}{g^\circ(\ol{\a}\c) g^\circ(\ol{\b_{i+1}\cdots\b_n}\c)}\FD{i}{\a\, ;\, \b_1, \dots,\b_i}{\ol{\b_{i+1}\cdots\b_n}\c}{\l_1,\dots, \l_i}.
\end{align*}
In particular, if $\b_1 \cdots \b_n \not\in \{\e, \ol{\a}\c \}$, we have
$$\FD{n}{\a\, ;\, \b_1,\dots,\b_n}{\c}{1,\dots,1} = \dfrac{g^\circ(\c) g(\ol{\a\b_{1}\cdots\b_n}\c)}{g^\circ(\ol{\a}\c) g^\circ(\ol{\b_{1}\cdots\b_n}\c)}.$$
\end{enumerate}
\end{thm}

\begin{rem}
For the case when $n=1$ of (ii), Otsubo \cite[Theorem 4.3 (i)]{Otsubo} gives the formula
\begin{equation}
\hF{2}{1}{\a,\b}{\c}{1} = \dfrac{g^\circ(\c) g(\ol{\a\b}\c)}{g^\circ(\ol{\a}\c) g^\circ(\ol{\b}\c)}, \label{Euler-Gauss}
\end{equation}
when $\{\a, \b\} \neq \{\e, \c\}$.
For general $n$, the latter half of (i) and (ii) are essentially given by He \cite[Theorem 4.7]{He}. 
\end{rem}

\begin{proof}[Proof of Theorem \ref{FD red.}]

(i) By Lemma \ref{multinomial}, we have
\begin{align*}
& \FD{n}{\a\, ;\, \b_1,\dots,\b_n}{\c}{\l_1,\dots,\l_i,x,\dots,x}\\
& = \dfrac{1}{(1-q)^n} \sum_{\nu_1,\dots,\nu_n}\dfrac{(\a)_{\nu_1\cdots\nu_n}\prod_{j=1}^n (\b_j)_{\nu_j}}{(\c)_{\nu_1\cdots\nu_n}^\circ \prod_{j=1}^n (\e)_{\nu_j}^\circ} \nu_{i+1}\cdots\nu_n (x) \prod_{j=1}^i \nu_j(\l_j)\\
& =  \dfrac{1}{(1-q)^n} \sum_{\nu_1,\dots,\nu_i,\nu}\dfrac{(\a)_{\nu_1\cdots\nu_i\nu} \prod_{j=1}^i (\b_j)_{\nu_j}}{(\c)_{\nu_1\cdots\nu_i\nu}^\circ \prod_{j=1}^i (\e)_{\nu_j}^\circ}\nu(x)\prod_{j=1}^i \nu_j(\l_j) \sum_{\nu_{i+1}\cdots\nu_n=\nu} \prod_{j=i+1}^n \dfrac{(\b_j)_{\nu_j}}{(\e)_{\nu_j}^\circ}\\
& = \dfrac{1}{(1-q)^{i+1}} \sum_{\nu_1,\dots,\nu_i,\nu}\dfrac{(\a)_{\nu_1\cdots\nu_i\nu} \prod_{j=1}^i (\b_j)_{\nu_j} (\b_{i+1}\cdots\b_n)_\nu }{(\c)_{\nu_1\cdots\nu_i\nu}^\circ \prod_{j=1}^i (\e)_{\nu_j}^\circ (\e)_\nu^\circ }\nu(x)\prod_{j=1}^i \nu_j(\l_j) \\
& = \FD{i+1}{\a\, ;\, \b_1,\dots, \b_i, \b_{i+1}\cdots\b_n}{\c}{\l_1,\dots,\l_i,x}.
\end{align*}
The particular case can be obtained by $i=0$.
(ii) By (i) and \eqref{Poch formula}, we have
\begin{align*}
& \FD{n}{\a\, ;\, \b_1,\dots,\b_n}{\c}{\l_1,\dots,\l_i,1,\dots,1}\\
& = \FD{i+1}{\a\, ;\, \b_1,\dots,\b_i,\b_{i+1}\cdots\b_n}{\c}{\l_1,\dots,\l_i,1}\\
& = \dfrac{1}{(1-q)^{i}}\sum_{\nu_1,\dots,\nu_i}\dfrac{(\a)_{\nu_1\cdots\nu_i} \prod_{j=1}^i (\b_j)_{\nu_j}}{(\c)_{\nu_1\cdots\nu_i}^\circ \prod_{j=1}^i (\e)_{\nu_j}^\circ}\prod_{j=1}^i \nu_j(\l_j) \hF{2}{1}{\a\nu_1\cdots\nu_i, \b_{i+1}\cdots\b_n}{\c\nu_1\cdots\nu_i}{1}.
\end{align*}
By \eqref{Euler-Gauss} (note that $\{\a\nu_1\cdots\nu_i, \b_{i+1}\cdots\b_n\} \neq \{\e, \c\nu_1\cdots\nu_i\}$ by the assumption), 
\begin{align*}
\hF{2}{1}{\a\nu_1\cdots\nu_i, \b_{i+1}\cdots\b_n}{\c\nu_1\cdots\nu_i}{1} = \dfrac{g^\circ(\c\nu_1\cdots\nu_i) g(\ol{\a\b_{i+1}\cdots\b_n}\c)}{g^\circ(\ol{\a}\c) g^\circ(\ol{\b_{i+1}\cdots\b_n}\c\nu_1\cdots\nu_i)}.
\end{align*}
Thus, noting that
$$\dfrac{g^\circ(\c\nu_1\cdots\nu_i) g(\ol{\a\b_{i+1}\cdots\b_n}\c)}{g^\circ(\ol{\a}\c) g^\circ(\ol{\b_{i+1}\cdots\b_n}\c\nu_1\cdots\nu_i)} = \dfrac{g^\circ(\c) g(\ol{\a\b_{i+1}\cdots\b_n}\c)}{g^\circ(\ol{\a}\c) g^\circ(\ol{\b_{i+1}\cdots\b_n}\c)} \cdot \dfrac{(\c)_{\nu_1\cdots\nu_i}^\circ}{(\ol{\b_{i+1}\cdots\b_n}\c)_{\nu_1\cdots\nu_i}^\circ},$$
we obtain the theorem.
The particular case can be obtained by $i=0$.
\end{proof}

\section*{Acknowledgements}
The author would like to thank Noriyuki Otsubo, Ryojun Ito, Satoshi Kumabe for their helpful comments.
The author also would like to thank Shunya Adachi for the useful discussion and helpful comments about Appell's functions over the complex numbers.
This work was supported by JST FOREST Program (JPMJFR2235).

\section*{Conflict of interest}
The author declares that there is no conflict of interest of any kind in the context of this paper.

\section*{Date availability}
Data sharing not applicable to this article as no datasets were generated or analysed during the current study.

\end{document}